\documentclass[12pt]{article} 
\usepackage{amsfonts,amsmath,latexsym,amssymb,mathrsfs,amsthm,comment}
\usepackage{slashbox}
\usepackage{caption}

\let\OLDthebibliography\thebibliography
\renewcommand\thebibliography[1]{
  \OLDthebibliography{#1}
  \setlength{\parskip}{3pt}
  \setlength{\itemsep}{0pt plus 0.3ex}
}

\def\numberlikeadb{\global\def\theequation{\thesection.\arabic{equation}}}
\numberlikeadb
\newtheorem{theorem}{Theorem}[section]
\newtheorem{lemma}[theorem]{Lemma}

\newtheorem{corollary}[theorem]{Corollary}

\newtheorem{remark}[theorem]{Remark}

\usepackage{lscape}
\usepackage{caption}
\usepackage{multirow}
\begin{document}

\title{Bounds for modified Struve functions of the first kind and their ratios}
\author{Robert E. Gaunt\footnote{School of Mathematics, The University of Manchester, Manchester M13 9PL, UK}}

\date{\today} 
\maketitle

\vspace{-5mm}

\begin{abstract}We obtain a simple two-sided inequality for the ratio $\mathbf{L}_\nu(x)/\mathbf{L}_{\nu-1}(x)$ in terms of the ratio $I_\nu(x)/I_{\nu-1}(x)$, where $\mathbf{L}_\nu(x)$ is the modified Struve function of the first kind and $I_\nu(x)$ is the modified Bessel function of the first kind.  This result allows one to use the extensive literature on bounds for $I_\nu(x)/I_{\nu-1}(x)$ to immediately deduce bounds for $\mathbf{L}_\nu(x)/\mathbf{L}_{\nu-1}(x)$.  We note some consequences and obtain further bounds for $\mathbf{L}_\nu(x)/\mathbf{L}_{\nu-1}(x)$ by adapting techniques used to bound the ratio $I_\nu(x)/I_{\nu-1}(x)$.  We apply these results to obtain new bounds for the condition numbers $x\mathbf{L}_\nu'(x)/\mathbf{L}_\nu(x)$, the ratio $\mathbf{L}_\nu(x)/\mathbf{L}_\nu(y)$ and the modified Struve function $\mathbf{L}_\nu(x)$ itself.  Amongst other results, we obtain two-sided inequalities for $x\mathbf{L}_\nu'(x)/\mathbf{L}_\nu(x)$ and $\mathbf{L}_\nu(x)/\mathbf{L}_\nu(y)$ that are given in terms of $xI_\nu'(x)/I_\nu(x)$ and $I_\nu(x)/I_\nu(y)$, respectively, which again allows one to exploit the substantial literature on bounds for these quantities.  The results obtained in this paper complement and improve existing bounds in the literature. 
\end{abstract}


\noindent{{{\bf{AMS 2010 Subject Classification:}}} Primary 33C20; 26D07. Secondary 33C10

\section{Introduction}\label{intro}

The ratios of modified Bessel functions $I_{\nu}(x)/I_{\nu-1}(x)$ and $K_{\nu-1}(x)/K_{\nu}(x)$ arise in many areas of the applied sciences, including epidemiology \cite{nh73}, chemical kinetics \cite{lbf03} and signal processing \cite{ks08}; see \cite{segura} and references therein for further applications.  These ratios are also key computational tools in the construction of numerical algorithms for computing modified Bessel functions (see, for example, Algorithms 12.6 and 12.7 of \cite{ast07}).  There is now an extensive literature on lower and upper bounds for these ratios; see \cite{amos74,b15,g32,kg13,ifantis,il07,im78, ln10,lorch, nasell,nasell2,pm50, rs16,segura,s12,soni}.  There is also a considerable literature on lower and upper bounds for the ratios $I_\nu(x)/I_\nu(y)$ and $K_\nu(x)/K_\nu(y)$; see \cite{amos74, b09, baricz2, bs09, bord, ifantis, jb96, jbb, l91,lm89, paris, ross73, si95}, which has been used, for example, to obtain tight bounds for the generalized Marcum Q-function, which arises in radar signal processing \cite{b09, bs09}.  

The modified Struve functions are related to the modified Bessel functions.  Likewise, they arise in manyfold applications, including leakage inductance in transformer windings \cite{hw94}, perturbation approximations of lee waves in a stratified flow \cite{mh69}, scattering of plane waves by circular cylinders \cite{s84} and lift and downwash distributions of oscillating wings in subsonic and supersonic flow \cite{w55,w54}; see \cite{bp13} for a list of further application areas.  

The first detailed study of inequalities for modified Struve functions was \cite{jn98}, in which two-sided inequalities for modified Struve functions and their ratios were obtained, together with Tur\'{a}n and Wronski type inequalities.  Recently, \cite{bp14} used a classical result on the monotonicity of quotients of Maclaurin series and techniques developed in the extensive study of modified Bessel functions and their ratios to obtain monotonicity results and, as a consequence, functional inequalities for the modified Struve function of the first kind $\mathbf{L}_\nu(x)$  that complement and improve the results of \cite{jn98}.  Further results and a new proof of a Tur\'{a}n-type inequality for the modified Struve function of the first kind are given in \cite{bps17}, and  monotonicity results and functional inequalities for the modified Struve function of the second kind $\mathbf{M}_\nu(x)=\mathbf{L}_\nu(x)-I_\nu(x)$ are given in \cite{bp142}.  It should be noted that the techniques used in \cite{bp14} and \cite{bp142} to obtain functional inequalities for $\mathbf{L}_\nu(x)$ and $\mathbf{M}_\nu(x)$ are quite different (this is also commented on in \cite{g18}), which is in contrast to the literature on modified Bessel functions in which functional inequalities for $I_\nu(x)$ and $K_\nu(x)$ are often developed in parallel.  For this reason, in this paper, with the exception of Remark \ref{rtyy}, we restrict our attention to the modified Struve function $\mathbf{L}_\nu(x)$.  

In this paper, we obtain new bounds for the ratios $\mathbf{L}_\nu(x)/\mathbf{L}_{\nu-1}(x)$ and $\mathbf{L}_\nu(x)/\mathbf{L}_{\nu}(y)$, the condition numbers $x\mathbf{L}_\nu'(x)/\mathbf{L}_\nu(x)$ and the modified Struve function $\mathbf{L}_\nu(x)$ itself.  These results complement and, at least in some cases, improve those given in \cite{bp14,jn98}.  Our approach is quite different, though.  In Section \ref{sec2}, we obtain a simple but accurate two-sided inequality for the ratio $\mathbf{L}_\nu(x)/\mathbf{L}_{\nu-1}(x)$ in terms of the ratio $I_\nu(x)/I_{\nu-1}(x)$.  This result is quite powerful because it allows one to exploit the extensive literature on bounds for $I_\nu(x)/I_{\nu-1}(x)$ to bound $\mathbf{L}_\nu(x)/\mathbf{L}_{\nu-1}(x)$.  We give some examples, and complement these bounds by showing that some of the techniques from the literature used to bound $I_\nu(x)/I_{\nu-1}(x)$ can be easily adapted to bound the ratio $\mathbf{L}_\nu(x)/\mathbf{L}_{\nu-1}(x)$.  In Section \ref{sec3}, we apply these bounds to obtain new bounds for the quantities $x\mathbf{L}_\nu'(x)/\mathbf{L}_\nu(x)$, $\mathbf{L}_\nu(x)/\mathbf{L}_\nu(y)$ and the modified Struve function $\mathbf{L}_\nu(x)$.  Amongst other results, we obtain two-sided inequalities for $x\mathbf{L}_\nu'(x)/\mathbf{L}_\nu(x)$ and $\mathbf{L}_\nu(x)/\mathbf{L}_\nu(y)$ that are given in terms of $xI_\nu'(x)/I_\nu(x)$ and $I_\nu(x)/I_\nu(y)$, respectively, which again allows one to exploit the substantial literature on these quantities.  Through a combination of asymptotic analysis of the bounds and numerical results, we find that, in spite of their simple form, the bounds obtained in this paper are quite accurate and often tight in certain limits.

\section{Upper and lower bounds for the ratio $\mathbf{L}_\nu(x)/\mathbf{L}_{\nu-1}(x)$}\label{sec2}

\subsection{Bounding $\mathbf{L}_\nu(x)/\mathbf{L}_{\nu-1}(x)$ via bounds for $I_\nu(x)/I_{\nu-1}(x)$}\label{sec2.2}

In this section, we obtain a simple but accurate double inequality for the ratio $\mathbf{L}_\nu(x)/\mathbf{L}_{\nu-1}(x)$ in terms of the ratio $I_\nu(x)/I_{\nu-1}(x)$.  The modified Bessel and modified Struve functions $I_\nu(x)$ and $\mathbf{L}_\nu(x)$ are closely related functions that have the following power series representations (see \cite{olver} for these and the forthcoming properties): 
\begin{align}\label{idefn}I_\nu(x)&=\sum_{n=0}^\infty\frac{(\frac{1}{2}x)^{2n+\nu}}{n!\Gamma(n+\nu+1)}, \\
\label{ldefn}\mathbf{L}_\nu(x)&=\sum_{n=0}^\infty\frac{(\frac{1}{2}x)^{2n+\nu+1}}{\Gamma(n+\frac{3}{2})\Gamma(n+\nu+\frac{3}{2})}.
\end{align}
It is immediate from these series representations, and the standard formulas $\Gamma(\frac{3}{2})=\frac{\sqrt{\pi}}{2}$ and $t\Gamma(t)=\Gamma(t+1)$ that, as $x\downarrow0$,
\begin{align}\label{itend0}I_\nu(x)&\sim \frac{x^\nu}{2^\nu\Gamma(\nu+1)}, \quad \nu>-1,\\
\label{ltend0}\mathbf{L}_\nu(x)&\sim \frac{x^{\nu+1}}{\sqrt{\pi}2^{\nu}\Gamma(\nu+\frac{3}{2})}\bigg(1+\frac{x^2}{3(2\nu+3)}\bigg), \quad \nu>-\tfrac{3}{2},
\end{align}
and both functions have very similar behaviour as $x\rightarrow\infty$:
\begin{align}\label{iinfty}I_\nu(x)&\sim\frac{\mathrm{e}^x}{\sqrt{2\pi x}}\bigg(1-\frac{4\nu^2-1}{8x}+\frac{(4\nu^2-1)(4\nu^2-9)}{128x^2}\bigg), \quad \nu\in\mathbb{R}, \\
\label{linfty}\mathbf{L}_\nu(x)&\sim\frac{\mathrm{e}^x}{\sqrt{2\pi x}}\bigg(1-\frac{4\nu^2-1}{8x}+\frac{(4\nu^2-1)(4\nu^2-9)}{128x^2}\bigg),\quad \nu\in\mathbb{R}.
\end{align}
The modified Struve function $\mathbf{L}_\nu(x)$ satisfies the relations
\begin{align}\label{struveid1}\mathbf{L}_{\nu-1}(x)-\mathbf{L}_{\nu+1}(x)&=\frac{2\nu}{x}\mathbf{L}_\nu(x)+a_\nu(x), \\
\label{struveid2}\mathbf{L}_{\nu-1}(x)+\mathbf{L}_{\nu+1}(x)&=2\mathbf{L}_\nu'(x)-a_\nu(x),
\end{align}
where $a_\nu(x)=\frac{(\frac{1}{2}x)^\nu}{\sqrt{\pi}\Gamma(\nu+\frac{3}{2})}$. The modified Bessel function $I_\nu(x)$ satisfies the same relations but without the $a_\nu(x)$ term.  For $\nu>-\frac{3}{2}$, it will be useful to define the function $b_\nu(x):(0,\infty)\rightarrow(0,\frac{1}{2})$ by
\begin{equation}\label{bdefn}b_\nu(x):=\frac{xa_\nu(x)}{2\mathbf{L}_\nu(x)}=\frac{(\frac{1}{2}x)^{\nu+1}}{\sqrt{\pi}\Gamma(\nu+\frac{3}{2})\mathbf{L}_\nu(x)}.
\end{equation}
This function will appear throughout this paper, and we collect some useful basic properties in the following lemma.

\begin{lemma}\label{blemma}The following assertions are true:

\vspace{2mm}

\noindent (i) For $\nu>-\frac{3}{2}$,
\begin{align}\label{bnu0}b_\nu(x)&\sim\frac{1}{2}-\frac{x^2}{6(2\nu+3)}, \quad x\downarrow0, \\
\label{bnu1}b_\nu(x)&\sim \frac{x^{\nu+3/2}\mathrm{e}^{-x}}{2^{\nu+1/2}\Gamma(\nu+\frac{3}{2})}, \quad x\rightarrow\infty.
\end{align}

\vspace{2mm}

\noindent (ii) For fixed $x>0$, $b_\nu(x)$ increases as $\nu$ increases in the interval $(-\frac{3}{2},\infty)$. 

\vspace{2mm}

\noindent (iii) If $\nu>-\frac{3}{2}$, then $b_\nu(x)$ is a decreasing function of $x$ in $(0,\infty)$.  Therefore, for $x>0$,
\begin{equation*}b_\nu(x)<b_\nu(0^+)=\frac{1}{2}, \quad \nu>-\tfrac{3}{2}.
\end{equation*}
This inequality can be improved further to
\begin{equation}\label{bcrude2}b_\nu(x)<\frac{1}{2}\bigg(1+\frac{x^2}{3(2\nu+3)}\bigg)^{-1}, \quad \nu>-\tfrac{3}{2}.
\end{equation}

\noindent (iv) For $x>0$, 
\begin{equation}\frac{x}{2}\mathrm{csch}(x)\label{star5}\leq b_\nu(x)<\frac{x}{4}\mathrm{csch}\bigg(\frac{x}{2\nu+3}\bigg), 
\end{equation}
where the lower and upper bounds are valid for $\nu\geq-\frac{1}{2}$ and $\nu>-1$, respectively.  We have equality in the lower bound if and only if $\nu=-\frac{1}{2}$ and the inequality is reversed if $-\frac{3}{2}<\nu<\frac{1}{2}$.
\end{lemma}


\begin{proof}(i) The expansion (\ref{bnu0}) can be obtained by using the asymptotic expansion (\ref{ltend0}), and (\ref{bnu1}) follows because $\mathbf{L}_\nu(x)\sim\frac{\mathrm{e}^x}{\sqrt{2\pi x}}$, as $x\rightarrow\infty$.

\vspace{2mm}

\noindent (ii) This is immediate from part (v) of Theorem 2.2 of \cite{bp14}.  

\vspace{2mm} 

\noindent (iii) It is clear from the series representation (\ref{ldefn}) that $1/b_\nu(x)$ is an increasing function of $x$ in $(0,\infty)$, and so $b_\nu(x)$ is a decreasing function of $x$ in $(0,\infty)$.  That $b_\nu(0^+)=\frac{1}{2}$ also follows from (\ref{ldefn}).  Inequality (\ref{bcrude2}) follows from truncating the series expansion of $\mathbf{L}_\nu(x)$ at the second term.

\vspace{2mm}

\noindent (iv) It was shown in \cite{bp14} that
\begin{equation*}\mathbf{L}_\nu(x)>\frac{x^\nu\sinh\big(\frac{x}{2\nu+3}\big)}{\sqrt{\pi}2^{\nu-1}\Gamma(\nu+\frac{3}{2})},\quad\nu>-1 \quad \text{and} \quad \mathbf{L}_{\nu}(x)\leq \frac{x^\nu\sinh(x)}{\sqrt{\pi}2^\nu\Gamma(\nu+\frac{3}{2})}, \quad \nu\geq-\tfrac{1}{2},
\end{equation*}
where there is equality in the second inequality if and only if $\nu=-\frac{1}{2}$ and the inequality is reversed if $-\frac{3}{2}<\nu<\frac{1}{2}$.  Applying these inequalities to (\ref{bdefn}) yields inequality (\ref{star5}).
\end{proof}

We now move on to the problem of bounding the ratio $\mathbf{L}_\nu(x)/\mathbf{L}_{\nu-1}(x)$.  To this end, we prove the following theorem, which gives a two-sided inequality for $\mathbf{L}_\nu(x)/\mathbf{L}_{\nu-1}(x)$ in terms of the ratio $I_\nu(x)/I_{\nu-1}(x)$.

\begin{theorem}\label{thmil}(i) For $x>0$,
\begin{align}\label{bounddo}I_{\nu}(x)\mathbf{L}_{\nu-1}(x)-I_{\nu-1}(x)\mathbf{L}_{\nu}(x)>0, \quad \nu\geq\tfrac{1}{2},
\end{align}
and 
\begin{align}\label{bound456}I_{\nu}(x)\mathbf{L}_{\nu-1}(x)-I_{\nu-1}(x)\mathbf{L}_{\nu}(x)& <\frac{(\frac{1}{2}x)^{\nu}I_{\nu}(x)}{\sqrt{\pi}\Gamma(\nu+\frac{3}{2})}, \quad \nu\geq-\tfrac{1}{2},\\
\label{bound789}I_{\nu}(x)\mathbf{L}_{\nu-1}(x)-I_{\nu-1}(x)\mathbf{L}_{\nu}(x)&<\frac{(\frac{1}{2}x)^{\nu-1}I_{\nu-1}(x)}{\sqrt{\pi}\Gamma(\nu+\frac{1}{2})}, \quad \nu\geq\tfrac{3}{2}.
\end{align}

\vspace{2mm}

\noindent (ii) For $x>0$,
\begin{equation}\label{firstcor1}\bigg(\frac{I_{\nu-1}(x)}{I_\nu(x)}+\frac{2b_\nu(x)}{x}\bigg)^{-1}<\frac{\mathbf{L}_{\nu}(x)}{\mathbf{L}_{\nu-1}(x)}<\frac{I_{\nu}(x)}{I_{\nu-1}(x)}, 
\end{equation}
where the lower bound is valid for $\nu\geq0$ and the upper bound is valid for $\nu\geq\frac{1}{2}$.  
\end{theorem}

In our proof, we shall make use of the following standard result on stochastic ordering of random variables \cite{stoch}.

\begin{lemma}\label{dfbgdf}Let $X$ and $Y$ be real-valued random variables.  Suppose that $\mathbb{P}(X>x)\leq\mathbb{P}(Y>x)$ for all $x\in\mathbb{R}$, and that additionally $\mathbb{P}(X>y)<\mathbb{P}(Y>y)$ for some $y\in\mathbb{R}$.  That is $X$ is \emph{stochastically strictly less than} $Y$.  Then, for all bounded, strictly increasing functions $f:\mathbb{R}\rightarrow\mathbb{R}$,
\begin{equation*}\mathbb{E}[f(X)]<\mathbb{E}[f(Y)].
\end{equation*}
In particular, if additionally $X$ and $Y$ have bounded support, then
\begin{equation*}\mathbb{E}[X^2]<\mathbb{E}[Y^2].
\end{equation*}
\end{lemma}

\noindent{\emph{Proof of Theorem \ref{thmil}}.} (i) We first prove inequality (\ref{bounddo}).  Suppose $\nu>\frac{1}{2}$; we shall deal with the case $\nu=\frac{1}{2}$ later.  From \cite[10.32.2, 11.5.6]{olver} we have the integral representations, for $\nu>-\frac{1}{2}$,
\begin{align}\label{Iint5}I_\nu(x)&=\frac{2(\frac{1}{2}x)^\nu}{\sqrt{\pi}\Gamma(\nu+\frac{1}{2})}\int_0^1(1-t^2)^{\nu-\frac{1}{2}}\cosh(xt)\,\mathrm{d}t , \\
\label{Lint5}\mathbf{L}_{\nu}(x)&=\frac{2(\frac{1}{2}x)^\nu}{\sqrt{\pi}\Gamma(\nu+\frac{1}{2})}\int_0^1(1-t^2)^{\nu-\frac{1}{2}}\sinh(xt)\,\mathrm{d}t  .
\end{align}
We thus obtain, for $\nu>\frac{1}{2}$,
\begin{align*}&I_{\nu}(x)\mathbf{L}_{\nu-1}(x)-I_{\nu-1}(x)\mathbf{L}_{\nu}(x) \\
&=\frac{4(\frac{1}{2}x)^{2\nu-1}}{\sqrt{\pi}\Gamma(\nu-\frac{1}{2})\Gamma(\nu+\frac{1}{2})}\bigg[\int_0^1\!\int_0^1(1-t^2)^{\nu-\frac{1}{2}}(1-s^2)^{\nu-\frac{3}{2}}\cosh(xt)\sinh(xs)\,\mathrm{d}t\,\mathrm{d}s \\
&\quad-\int_0^1\!\int_0^1(1-t^2)^{\nu-\frac{3}{2}}(1-s^2)^{\nu-\frac{1}{2}}\cosh(xt)\sinh(xs)\,\mathrm{d}t\,\mathrm{d}s\bigg] \\
&=\frac{4(\frac{1}{2}x)^{2\nu-1}}{\sqrt{\pi}\Gamma(\nu-\frac{1}{2})\Gamma(\nu+\frac{1}{2})}\int_0^1\!\int_0^1(1-t^2)^{\nu-\frac{3}{2}}(1-s^2)^{\nu-\frac{3}{2}}(s^2-t^2)\cosh(xt)\sinh(xs)\,\mathrm{d}t\,\mathrm{d}s \\
&=\frac{4(\frac{1}{2}x)^{2\nu-1}}{\sqrt{\pi}\Gamma(\nu-\frac{1}{2})\Gamma(\nu+\frac{1}{2})}\bigg[\int_0^1f_\nu(x,t)\,\mathrm{d}t\int_0^1s^2g_\nu(x,s)\,\mathrm{d}s-\int_0^1t^2f_\nu(x,t)\,\mathrm{d}t\int_0^1g_\nu(x,s)\,\mathrm{d}s\bigg],
\end{align*}
where
\begin{align*}f_\nu(x,t):=(1-t^2)^{\nu-\frac{3}{2}}\cosh(xt), \quad g_\nu(x,t):=(1-t^2)^{\nu-\frac{3}{2}}\sinh(xt).
\end{align*}
Proving that $I_{\nu}(x)\mathbf{L}_{\nu-1}(x)-I_{\nu-1}(x)\mathbf{L}_{\nu}(x)>0$ for all $x>0$ is thus equivalent to proving that 
\begin{align}\label{stovb}\frac{\int_0^1 t^2f_\nu(x,t)\,\mathrm{d}t}{\int_0^1f_\nu(x,t)\,\mathrm{d}t}<\frac{\int_0^1 t^2g_\nu(x,t)\,\mathrm{d}t}{\int_0^1g_\nu(x,t)\,\mathrm{d}t}, \quad \text{for all $x>0$.}
\end{align}
Now let $X$ and $Y$ be random variables supported on $(0,1)$, with probability density functions $f_\nu(x,t)/\int_0^1f_\nu(x,s)\,\mathrm{d}s$ and $g_\nu(x,t)/\int_0^1g_\nu(x,s)\,\mathrm{d}s$, respectively.  Therefore (\ref{stovb}) can be written as
\begin{equation}\label{bvhsb6}\mathbb{E}[X^2]<\mathbb{E}[Y^2].
\end{equation}
But, since, for fixed $x>0$ and $\nu>\frac{1}{2}$, the ratio $f_\nu(x,t)/g_\nu(x,t)$ is strictly decreasing in $t$ on the interval $(0,1)$ (in the probability literature, one would say that $X$ is less than $Y$ according to likelihood ratio ordering \cite{stoch}), we have, for $z\in(0,1)$,
\begin{align*}\mathbb{P}(X>z)=\frac{\int_z^1 f_\nu(x,t)\,\mathrm{d}t}{\int_0^1f_\nu(x,t)\,\mathrm{d}t}<\frac{\int_z^1 g_\nu(x,t)\,\mathrm{d}t}{\int_0^1g_\nu(x,t)\,\mathrm{d}t}=\mathbb{P}(Y>z).
\end{align*}
  Therefore (\ref{bvhsb6}) holds due to Lemma \ref{dfbgdf}, meaning that $I_{\nu}(x)\mathbf{L}_{\nu-1}(x)-I_{\nu-1}(x)\mathbf{L}_{\nu}(x)>0$ for all $x>0$, if $\nu>\frac{1}{2}$.

Now, we consider the case $\nu=\frac{1}{2}$.  From \cite[10.49(ii), 11.4(i)]{olver} we have the formulas
\begin{align}&I_{\frac{1}{2}}(x)=\sqrt{\frac{2}{\pi x}}\sinh(x), \quad \mathbf{L}_{\frac{1}{2}}(x)=\sqrt{\frac{2}{\pi x}}\big(\cosh(x)-1\big), \nonumber \\
\label{sph9}&I_{-\frac{1}{2}}(x)=\sqrt{\frac{2}{\pi x}}\cosh(x), \quad \mathbf{L}_{-\frac{1}{2}}(x)=\sqrt{\frac{2}{\pi x}}\sinh(x).
\end{align} 
Using these formulas and the standard identity $\cosh^2(x)-\sinh^2(x)=1$ we have that
\[I_{\frac{1}{2}}(x)\mathbf{L}_{-\frac{1}{2}}(x)-I_{-\frac{1}{2}}(x)\mathbf{L}_{\frac{1}{2}}(x)=\frac{2}{\pi x}\big(\cosh(x)-1\big)>0.\]

We now deduce inequalities (\ref{bound456}) and (\ref{bound789}) from (\ref{bounddo}).  Substituting the relations
\begin{align*}I_{\nu+1}(x)=I_{\nu-1}(x)-\frac{2\nu}{x}I_{\nu}(x),\quad \mathbf{L}_{\nu+1}(x)=\mathbf{L}_{\nu-1}(x)-\frac{2\nu}{x}\mathbf{L}_{\nu}(x)-\frac{(\frac{1}{2}x)^{\nu}}{\sqrt{\pi} \Gamma(\nu+\frac{3}{2})}
\end{align*}
into the inequality $I_{\nu+1}(x)\mathbf{L}_{\nu}(x)-I_{\nu}(x)\mathbf{L}_{\nu+1}(x)>0$ gives inequality (\ref{bound456}).  Similarly, on substituting the relations
\begin{align*}I_{\nu-1}(x)=I_{\nu+1}(x)+\frac{2\nu}{x}I_{\nu}(x), \quad \mathbf{L}_{\nu-1}(x)=\mathbf{L}_{\nu+1}(x)+\frac{2\nu}{x}\mathbf{L}_\nu(x)+\frac{(\frac{1}{2}x)^{\nu}}{\sqrt{\pi} \Gamma(\nu+\frac{3}{2})}
\end{align*}
into the inequality $I_{\nu}(x)\mathbf{L}_{\nu-1}(x)-I_{\nu-1}(x)\mathbf{L}_{\nu}(x)>0,$ and then replacing $\nu$ by $\nu-1$, we deduce inequality (\ref{bound789}). 

\vspace{2mm}

\noindent (ii) The two-sided inequality follows from rearranging inequalities (\ref{bounddo}) and (\ref{bound456}) using the facts that $I_\nu(x)>0$ and $\mathbf{L}_\nu(x)>0$ for all $x>0$ if $\nu\geq-1$, and that $b_\nu(x)=\frac{xa_\nu(x)}{2\mathbf{L}_\nu(x)}$.  \hfill $\Box$


\begin{remark}
Numerical experiments carried out with \emph{Mathematica} suggest that inequality (\ref{bounddo}) of Theorem \ref{thmil} holds for all $\nu\geq-\frac{1}{2}$ (which would mean that the upper bound of (\ref{firstcor1}) would hold for all $\nu\geq0$).  It should be noted that the parameter range in the two-sided inequality (\ref{firstcor1}) is sufficient for the purposes of this paper, but extending the range may be useful in other applications.  An alternative method will be needed for $-\frac{1}{2}\leq\nu<\frac{1}{2}$, because the integral representations (\ref{Iint5}) and (\ref{Lint5}) are not valid for such $\nu$.  The case $\nu=-\frac{1}{2}$ is straightforward to verify, though.  From \cite[10.49(ii), 11.4(i)]{olver} we have the formulas
\begin{align*}I_{-\frac{3}{2}}(x)=\sqrt{\frac{2}{\pi x}}\bigg(\sinh(x)-\frac{\cosh(x)}{x}\bigg), \quad \mathbf{L}_{-\frac{3}{2}}(x)=\sqrt{\frac{2}{\pi x}}\bigg(\cosh(x)-\frac{\sinh(x)}{x}\bigg).
\end{align*}
Using (\ref{sph9}), these formulas and the standard identity $\cosh^2(x)-\sinh^2(x)=1$ gives
\[I_{-\frac{1}{2}}(x)\mathbf{L}_{-\frac{3}{2}}(x)-I_{-\frac{3}{2}}(x)\mathbf{L}_{-\frac{1}{2}}(x)=\frac{2}{\pi x}>0.\] 
\end{remark}


\begin{remark}\label{rtyy}We can use the formula $\mathbf{M}_\nu(x)=\mathbf{L}_\nu(x)-I_\nu(x)$ to write
\[I_{\nu}(x)\mathbf{L}_{\nu-1}(x)-I_{\nu-1}(x)\mathbf{L}_{\nu}(x)=I_{\nu}(x)\mathbf{M}_{\nu-1}(x)-I_{\nu-1}(x)\mathbf{M}_{\nu}(x).\]
Therefore we also have the double inequality 
\begin{equation*}0<I_{\nu}(x)\mathbf{M}_{\nu-1}(x)-I_{\nu-1}(x)\mathbf{M}_{\nu}(x)<\frac{(\frac{1}{2}x)^{\nu}I_{\nu}(x)}{\sqrt{\pi}\Gamma(\nu+\frac{3}{2})}, 
\end{equation*}
where the lower bound is valid for $\nu\geq\frac{1}{2}$ and the upper bound is valid for $\nu\geq-\frac{1}{2}$.  Now, if $\nu\geq-\frac{1}{2}$, we have $\mathbf{M}_\nu(x)<0$ for all $x>0$ (see \cite[11.5.4]{olver} for the case $\nu>-\frac{1}{2}$, and the formulas in (\ref{sph9}) for the case $\nu=-\frac{1}{2}$).  Therefore, rearranging the above lower bound gives that, for $x>0$,
\begin{equation*}\frac{\mathbf{M}_{\nu}(x)}{\mathbf{M}_{\nu-1}(x)}>\frac{I_{\nu}(x)}{I_{\nu-1}(x)}, \quad \nu\geq\tfrac{1}{2}.
\end{equation*}
Since $\mathbf{M}_\nu(x)=\mathbf{L}_\nu(x)-I_\nu(x)$ and $I_\nu(x)/\mathbf{L}_\nu(x)\gg1$, as $x\downarrow0$, for $\nu\geq-\frac{1}{2}$, it follows that this inequality is tight as $x\downarrow0$.  However, from the asymptotic formula \cite[11.6.2]{olver}
\begin{equation*}\mathbf{M}_\nu(x)\sim-\frac{(\frac{1}{2}x)^{\nu-1}}{\sqrt{\pi}\Gamma(\nu+\frac{1}{2})}, \quad \nu>-\tfrac{1}{2},
\end{equation*}
as $x\rightarrow\infty$, and the asymptotic formula (\ref{iinfty}), we have that, as $x\rightarrow\infty$, $\mathbf{M}_\nu(x)/\mathbf{M}_{\nu-1}(x)=O(x)$ for $\nu\geq-\frac{1}{2}$, whereas $I_\nu(x)/I_{\nu-1}(x)=O(1)$ for all $\nu\in\mathbb{R}$.  
\end{remark}

\begin{remark}Let $l_\nu^a(x)$, $u_\nu^a(x)$ denote the lower and upper bounds of the double inequality (\ref{firstcor1}) and let $h_\nu(x)=\mathbf{L}_\nu(x)/\mathbf{L}_{\nu-1}(x)$.  The double inequality (\ref{firstcor1}) is tight as $x\downarrow\infty$.  Indeed, from the asymptotic formulas (\ref{iinfty}) and (\ref{bnu1}), we have, as $x\rightarrow\infty$,
\begin{equation*}\frac{u_\nu^a(x)}{l_\nu^a(x)}-1=\frac{2b_\nu(x)}{x}\frac{I_\nu(x)}{I_{\nu-1}(x)}=O(x^{\nu+1/2}\mathrm{e}^{-x}).
\end{equation*}
From the asymptotic formulas (\ref{itend0}) and (\ref{ltend0}) and the fact that $b_\nu(0^+)=\frac{1}{2}$, we have
\begin{equation*}1-\lim_{x\downarrow0}\frac{l_\nu^a(x)}{h_\nu(x)}=0, \quad \lim_{x\downarrow0}\frac{u_\nu^a(x)}{h_\nu(x)}-1=\frac{1}{2\nu},
\end{equation*}
and so the relative error in approximating $h_\nu(x)$ by $l_\nu^a(x)$ is 0 in the limit $x\downarrow0$, and the relative error in approximating $h_\nu(x)$ by $u_\nu^a(x)$ in the limit $x\downarrow0$ decreases as $\nu$ increases, and the bound is tight as $\nu\rightarrow\infty$.  Further insight into the accuracy of these bounds can be gained from Tables \ref{table1} and \ref{table11}.  Despite their simple form, the bounds can be seen to be quite accurate, with a relative error of less than $0.01$ for both bounds when $x\geq10$ for the values of $\nu$ we considered.  This accuracy is a consequence of the exponential decay of $b_\nu(x)$.
\end{remark}

\begin{table}[h]
\centering
\caption{\footnotesize{Relative error in approximating $\mathbf{L}_{\nu}(x)/\mathbf{L}_{\nu-1}(x)$ by $\big(I_{\nu-1}(x)/I_\nu(x)+2b_\nu(x)/x\big)^{-1}$.}}
\label{table1}
{\scriptsize
\begin{tabular}{|c|rrrrrrrrr|}
\hline
 \backslashbox{$\nu$}{$x$}      &    0 &    0.5 &    1 &    2.5 &    5 &    7.5 & 10 & 15 & 25  \\
 \hline
$0$ & 0.0000 & 0.0355 & 0.0947 & 0.1073 & 0.0196 & 0.0022 & 0.0002 & 0.0000 & 0.0000 \\
$0.5$ & 0.0000 & 0.0097 & 0.0312 & 0.0623 & 0.0200 & 0.0030 & 0.0003 & 0.0000 & 0.0000 \\
1 & 0.0000 & 0.0040 & 0.0138 & 0.0377 & 0.0186 & 0.0037 & 0.0005 & 0.0000 & 0.0000 \\
2.5 & 0.0000 & 0.0007 & 0.0027 & 0.0112 & 0.0122 & 0.0047 & 0.0011 & 0.0000 & 0.0000 \\
5 & 0.0000 & 0.0001 & 0.0006 & 0.0028 & 0.0054 & 0.0039 & 0.0016 & 0.0001 & 0.0000 \\
7.5 & 0.0000 & 0.0000 & 0.0002 & 0.0011 & 0.0026 & 0.0026 & 0.0016 & 0.0002 & 0.0000 \\
10 & 0.0000 & 0.0000 & 0.0001 & 0.0005 & 0.0014   &  0.0018 & 0.0013 & 0.0003 & 0.0000 \\  
  \hline
\end{tabular}}
\end{table}

\begin{table}[h]
\centering
\caption{\footnotesize{Relative error in approximating $\mathbf{L}_{\nu}(x)/\mathbf{L}_{\nu-1}(x)$ by $I_{\nu}(x)/I_{\nu-1}(x)$.}}
\label{table11}
{\scriptsize
\begin{tabular}{|c|rrrrrrrrr|}
\hline
 \backslashbox{$\nu$}{$x$}      &    0 &    0.5 &    1 &    2.5 &    5 &    7.5 & 10 & 15 & 25  \\
 \hline
$0$ & $\infty$ & 7.7021 & 1.7232 & 0.1394 & 0.0061 & 0.0004 & 0.0000 & 0.0000 & 0.0000 \\
$0.5$ & 1.0000 & 0.8868 & 0.6481 & 0.1631 & 0.0135 & 0.0011 & 0.0001 & 0.0000 & 0.0000 \\
1 & 0.5000 & 0.4711 & 0.3981 & 0.1587 & 0.0206 & 0.0022 & 0.0002 & 0.0000 & 0.0000 \\
2.5 & 0.2000 & 0.1957 & 0.1833 & 0.1200 & 0.0344 & 0.0066 & 0.0010 & 0.0000 & 0.0000 \\
5 & 0.1000 & 0.0990 & 0.0961 & 0.0780 & 0.0387 & 0.0132 & 0.0034 & 0.0001 & 0.0000 \\
7.5 & 0.0667 & 0.0662 & 0.0650 & 0.0568 & 0.0354 & 0.0165 & 0.0059 & 0.0004 & 0.0000 \\
10 & 0.0500 & 0.0498 & 0.0491 & 0.0445 & 0.0313   &  0.0175 & 0.0078 & 0.0009 & 0.0000 \\  
  \hline
\end{tabular}}
\end{table}

The double inequality (\ref{firstcor1}) allows one to exploit the substantial literature on bounds for the ratio $I_\nu(x)/I_{\nu-1}(x)$ to bound the ratio $\mathbf{L}_\nu(x)/\mathbf{L}_{\nu-1}(x)$.  Because of the accuracy of the double inequality (\ref{firstcor1}), the accuracy of the resulting bounds for $\mathbf{L}_\nu(x)/\mathbf{L}_\nu(x)$ will be similar to that of the initial bounds for $I_\nu(x)/I_{\nu-1}(x)$.  In the following corollary, we note some examples.

\begin{corollary}\label{bghj}For $x>0$,
\begin{equation}\label{upper1}\frac{x}{\nu-\frac{1}{2}+2b_{\nu}(x)+\sqrt{\big(\nu+\frac{1}{2}\big)^2+x^2}}<\frac{\mathbf{L}_{\nu}(x)}{\mathbf{L}_{\nu-1}(x)}<\frac{x}{\nu-\frac{1}{2}+\sqrt{\big(\nu-\frac{1}{2}\big)^2+x^2}},
\end{equation}
where the lower bound holds for $\nu\geq0$ and the upper bound holds for $\nu\geq\frac{1}{2}$.  Also,
\begin{equation}\label{lowtanh}\frac{\mathbf{L}_{\nu}(x)}{\mathbf{L}_{\nu-1}(x)}>\frac{x\tanh(x)}{x+(2\nu-1)\tanh(x)+2b_\nu(x)\tanh(x)}, \quad \nu>\tfrac{1}{2}.
\end{equation}
\end{corollary}

\begin{proof}It was shown in \cite{segura} (see also \cite{amos74,ln10,rs16}) that, for $x>0$,
\[\frac{x}{\nu-\frac{1}{2}+\sqrt{\big(\nu+\frac{1}{2}\big)^2+x^2}}<\frac{I_{\nu}(x)}{I_{\nu-1}(x)}<\frac{x}{\nu-\frac{1}{2}+\sqrt{\big(\nu-\frac{1}{2}\big)^2+x^2}},\]
where the lower bound holds for $\nu\geq0$ and the upper bound holds for $\nu\geq\frac{1}{2}$, and it was shown in \cite{ifantis} that
\[\frac{I_{\nu}(x)}{I_{\nu-1}(x)}>\frac{x\tanh(x)}{x+(2\nu-1)\tanh(x)}, \quad \nu>\tfrac{1}{2}.\]
Now combine these bounds with inequality (\ref{firstcor1}) of Theorem \ref{thmil}.
\end{proof}

\begin{remark}The lower bound (\ref{lowtanh}) complements the following upper bound of \cite{bp14}:
\begin{equation}\label{bpineq3}\frac{\mathbf{L}_{\nu}(x)}{\mathbf{L}_{\nu-1}(x)}\leq\frac{\cosh(x)-1}{\sinh(x)}=\tanh\bigg(\frac{x}{2}\bigg), \quad 
x>0,\:\nu\geq\tfrac{1}{2},
\end{equation} 
where we have equality if and only if $\nu=\frac{1}{2}$. It should also be noted that a more complicated bound, valid for $\nu\geq\frac{3}{2}$, which improves on (\ref{bpineq3}) is given by inequality (3.1) of \cite{bp14}.

Let us compare our bound (\ref{upper1}) with (\ref{bpineq3}).  We used \emph{Mathematica} to observe that if $\nu>\frac{3}{2}$, then $x/({\nu-\frac{1}{2}+\sqrt{(\nu-\frac{1}{2})^2+x^2}})<\tanh\big(\frac{x}{2}\big)$ for all $x>0$. (We checked the case $\nu=\frac{3}{2}$, since if the inequality holds for this value of $\nu$ then it must hold for for all $\nu>\frac{3}{2}$.) An asymptotic analysis also shows that when $\frac{1}{2}<\nu<\frac{3}{2}$ inequality (\ref{bpineq3}) performs better in the limit $x\downarrow0$, whilst inequality (\ref{upper1}) is better as $x\rightarrow\infty$.  We used \emph{Mathematica} to find the value $x_\nu^*$ at which the two upper bounds are equal.  We find $x_{\frac{5}{8}}^*=4.21$, $x_{\frac{3}{4}}^*=3.26$, $x_{\frac{7}{8}}^*=2.66$, $x_{1}^*=2.18$, $x_{\frac{9}{8}}^*=1.76$, $x_{\frac{5}{4}}^*=1.35$, $x_{\frac{11}{8}}^*=0.91$.
\end{remark}

\begin{remark}Since $b_\nu(x)<\frac{1}{2}$ for all $x>0$, we have the following simpler two-sided inequality:
\begin{equation*}\frac{x}{\nu+\frac{1}{2}+\sqrt{\big(\nu+\frac{1}{2}\big)^2+x^2}}<\frac{\mathbf{L}_{\nu}(x)}{\mathbf{L}_{\nu-1}(x)}<\frac{x}{\nu-\frac{1}{2}+\sqrt{\big(\nu-\frac{1}{2}\big)^2+x^2}},
\end{equation*}
with the same range of validity as (\ref{upper1}).  Similar such simplifications can be made to all bounds given in this paper. 
\end{remark}

\begin{remark}We will use the lower and upper bounds of (\ref{upper1}) throughout this paper.  It is therefore useful to gain some insight into the quality of the approximation.  Denote the lower and upper bounds by $l_\nu^b(x)$ and $u_\nu^b(x)$, respectively, and write $h_\nu(x)=\mathbf{L}_\nu(x)/\mathbf{L}_{\nu-1}(x)$.  From the asymptotic formula (\ref{ltend0}) and the fact that $b_\nu(0^+)=\frac{1}{2}$, we have 
\begin{equation*}1-\lim_{x\downarrow0}\frac{l_\nu^b(x)}{h_\nu(x)}=0, \quad \lim_{x\downarrow0}\frac{u_\nu^b(x)}{h_\nu(x)}-1=\frac{2}{2\nu-1},
\end{equation*}
and so the relative error in approximating $h_\nu(x)$ by $l_\nu^b(x)$ is 0 in the limit $x\downarrow0$.  The relative error in approximating $h_\nu(x)$ by $u_\nu^b(x)$ in the limit $x\downarrow0$ blows up as $\nu\downarrow\frac{1}{2}$, but decreases as $\nu$ increases, and the bound is tight as $\nu\rightarrow\infty$.  Furthermore, as $x\downarrow0$,
\[l_\nu^b(x)\sim\frac{x}{2\nu+1}-\frac{4(\nu+2)x^3}{3(2\nu+1)^3(2\nu+3)}, \quad h_\nu(x)\sim\frac{x}{2\nu+1}-\frac{2x^3}{3(2\nu+1)^2(2\nu+3)},\]
and so, as $\nu\rightarrow\infty$, the second term in the $x\downarrow0$ expansion of $l_\nu^b(x)$ approaches the second term of the expansion of $h_\nu(x)$.  Hence, both the lower and upper bounds improve for `small' $x$ as $\nu$ increases.

From the asymptotic formulas (\ref{linfty}) and (\ref{bnu1}), we have, as $x\rightarrow\infty$,
\begin{align*}l_\nu^b(x)&\sim1-\frac{2\nu-1}{2x}+\frac{4\nu^2-4\nu+1}{8x^2}, \quad h_\nu(x)\sim1-\frac{2\nu-1}{2x}+\frac{4\nu^2-8\nu+3}{8x^2}, \\
 u_\nu^b(x)&\sim1-\frac{2\nu-1}{2x}+\frac{4\nu^2-12\nu+1}{8x^2}.
\end{align*}
All of $l_\nu^b(x)$, $h_\nu(x)$ and $u_\nu^b(x)$ have the same first two terms in the $x\rightarrow\infty$ expansion, but differ in the third term.  We see that, for `large' $x$, the quality of both the lower and upper bound approximations decreases as $\nu$ increases.  The $O(x^{-2})$ error in the approximation is much larger than the $O(x^{\nu+1/2}\mathrm{e}^{-x})$ error of the double inequality (\ref{firstcor1}).  The comments given in this remark are supported by numerical results obtained using \emph{Mathematica}, which are reported in Tables \ref{table3} and \ref{table4}.
\end{remark}

\begin{table}[h]
\centering
\caption{\footnotesize{Relative error in approximating $\mathbf{L}_{\nu}(x)/\mathbf{L}_{\nu-1}(x)$ by the lower bound of (\ref{upper1}).}}
\label{table3}
{\scriptsize
\begin{tabular}{|c|rrrrrrrrrr|}
\hline
 \backslashbox{$\nu$}{$x$}      &    0 &    0.5 &    1 &    2.5 &    5 &    7.5 & 10 & 15 & 25 & 50  \\
 \hline
0 & 0.0000 & 0.1057 & 0.1973 & 0.1545 & 0.0319 & 0.0073 & 0.0030 & 0.0012 & 0.0004 & 0.0001 \\
0.5 & 0.0000 & 0.0267 & 0.0732 & 0.1073 & 0.0383 & 0.0117 & 0.0053 & 0.0022 & 0.0008 & 0.0002 \\
1 & 0.0000 & 0.0102 & 0.0329 & 0.0725 & 0.0390 & 0.0147 & 0.0071 & 0.0031 & 0.0011 & 0.0003 \\
2.5 & 0.0000 & 0.0017 & 0.0063 & 0.0243 & 0.0287 & 0.0173 & 0.0100 & 0.0049 & 0.0020 & 0.0006 \\
5 & 0.0000 & 0.0003 & 0.0012 & 0.0062 & 0.0132   &  0.0128 & 0.0098 & 0.0059 & 0.0029 & 0.0009 \\
7.5 & 0.0000 & 0.0001 & 0.0004 & 0.0024 & 0.0063 & 0.0081 & 0.0076 & 0.0056 & 0.0033 & 0.0012 \\ 
10 & 0.0000 & 0.0000 & 0.0002  & 0.0011    & 0.0034     & 0.0051  & 0.0055 & 0.0048 & 0.0033 & 0.0014 \\  
  \hline
\end{tabular}}
\end{table}

\begin{table}[h]
\centering
\caption{\footnotesize{Relative error in approximating $\mathbf{L}_{\nu}(x)/\mathbf{L}_{\nu-1}(x)$ by the upper bound of (\ref{upper1}).}}
\label{table4}
{\scriptsize
\begin{tabular}{|c|rrrrrrrrrr|}
\hline
 \backslashbox{$\nu$}{$x$}      &    0 &    0.5 &    1 &    2.5 &    5 &    7.5 & 10 & 15 & 25 & 50  \\
 \hline
0.5 & $\infty$ & 3.0830 & 1.1640 & 0.1789 & 0.0136 & 0.0011 & 0.0001 & 0.0000 & 0.0000 & 0.0000 \\
1 & 2.0000 & 1.5128 & 0.9357 & 0.2417 & 0.0338 & 0.0074 & 0.0030 & 0.0012 & 0.0004 & 0.0001 \\
2.5 & 0.5000 & 0.4824 & 0.4360 & 0.2524 & 0.0777 & 0.0259 & 0.0117 & 0.0047 & 0.0017 & 0.0004 \\
5 & 0.2222 & 0.2199 & 0.2131 & 0.1736 & 0.0950   &  0.0460 & 0.0239 & 0.0099 & 0.0036 & 0.0009 \\
7.5 & 0.1429 & 0.1421 & 0.1397 & 0.1247 & 0.0864 & 0.0523 & 0.0310 & 0.0139 & 0.0054 & 0.0014 \\ 
10 & 0.1053 & 0.1049 & 0.1037  & 0.0962    & 0.0747     & 0.0516  & 0.0341 & 0.0166 & 0.0069 & 0.0019 \\  
  \hline
\end{tabular}}
\end{table}

\subsection{Further bounds for the ratio $\mathbf{L}_\nu(x)/\mathbf{L}_{\nu-1}(x)$}\label{sec2.4}

Since the modified Struve function $\mathbf{L}_\nu(x)$ and modified Bessel function $I_\nu(x)$ are closely related, some of the techniques used in the literature to bound the ratio $I_\nu(x)/I_{\nu-1}(x)$ can be easily adapted to bound the ratio $\mathbf{L}_\nu(x)/\mathbf{L}_{\nu-1}(x)$.  In the following theorem, we give two such examples that complement the bounds we gave in Corollary \ref{bghj}.  The first bound is obtained by adapting the method used to prove Theorem 1.1 of \cite{ln10}, and we adapt the approach used to prove inequality (1.9) of \cite{ifantis} to establish the second.

\begin{theorem}For $x>0$,
\begin{equation}\label{lower1}\frac{\mathbf{L}_{\nu}(x)}{\mathbf{L}_{\nu-1}(x)}>\frac{x}{\nu+b_\nu(x)+\sqrt{(\nu+b_\nu(x))^2+x^2}}, \quad \nu\geq-\tfrac{1}{2},
\end{equation}
and
\begin{equation}\label{bibidt}\frac{\mathbf{L}_{\nu}(x)}{\mathbf{L}_{\nu-1}(x)}>\frac{x\tanh\big(\frac{1}{2}x\big)}{x+(2\nu-1)\tanh\big(\frac{1}{2}x\big)+2\big(b_\nu(x)-b_{\frac{1}{2}}(x)\big)\tanh\big(\frac{1}{2}x\big)}, \quad \nu>\tfrac{1}{2}.
\end{equation}
We have equality in (\ref{bibidt}) if $\nu=\frac{1}{2}$.
\end{theorem}

\begin{proof}We begin by noting a Tur\'{a}n-type inequality, which was proved by \cite{bp14, jn98}.  For $x>0$ and $\nu>-\frac{3}{2}$, we have $\mathbf{L}_{\nu-1}(x)\mathbf{L}_{\nu+1}(x)<\mathbf{L}_\nu^2(x)$.  From the relation (\ref{struveid1}), we thus obtain
\begin{equation*}\mathbf{L}_{\nu-1}(x)\bigg[\mathbf{L}_{\nu-1}(x)-\frac{2\nu}{x}\mathbf{L}_\nu(x)-a_\nu(x)\bigg]<\mathbf{L}_\nu^2(x).
\end{equation*}
Dividing both sides by $\mathbf{L}_{\nu-1}^2(x)$ and defining $h_\nu(x)=\mathbf{L}_\nu(x)/\mathbf{L}_{\nu-1}(x)$, we obtain
\begin{equation*}1-\bigg(\frac{2\nu}{x}+\frac{2b_\nu(x)}{x}\bigg)h_\nu(x)<h_\nu^2(x).
\end{equation*}
Solving this quadratic inequality gives, for $\nu\geq-\frac{1}{2}$,
\begin{align*}h_\nu(x)&>-\frac{\nu}{x}-\frac{b_\nu(x)}{x}+\sqrt{\bigg(\frac{\nu}{x}+\frac{b_\nu(x)}{x}\bigg)^2+1} =\frac{x}{\nu+b_\nu(x)+\sqrt{(\nu+b_\nu(x))^2+x^2}}.
\end{align*}

Moving on to inequality (\ref{bibidt}), from relation (\ref{struveid1}), we have
\begin{equation*}\frac{\mathbf{L}_{\nu-1}(x)}{\mathbf{L}_{\nu}(x)}-\frac{2\nu}{x}-\frac{2b_\nu(x)}{x}=\frac{\mathbf{L}_{\nu+1}(x)}{\mathbf{L}_{\nu}(x)}.
\end{equation*}
Now, by part (vi) of Theorem 2.2 of \cite{bp14}, we have that, for all $x>0$, the function $\mathbf{L}_{\nu+1}(x)/\mathbf{L}_{\nu}(x)$ decreases as $\nu$ increases in the interval $(\frac{1}{2},\infty)$, and therefore the function $\frac{\mathbf{L}_{\nu-1}(x)}{\mathbf{L}_{\nu}(x)}-\frac{2\nu}{x}-\frac{2b_\nu(x)}{x}$ also decreases as $\nu$ increases in the interval $(\frac{1}{2},\infty)$.  Using the standard formulas (see \cite[11.4(i)]{olver})
\begin{equation*}\mathbf{L}_{-\frac{1}{2}}(x)=\sqrt{\frac{2}{\pi x}}\sinh(x), \quad \mathbf{L}_{\frac{1}{2}}(x)=\sqrt{\frac{2}{\pi x}}\big(\cosh(x)-1\big)
\end{equation*}
we have
\begin{equation*}\frac{\mathbf{L}_{-\frac{1}{2}}(x)}{\mathbf{L}_{\frac{1}{2}}(x)}=\frac{\sinh(x)}{\cosh(x)-1}=\coth\bigg(\frac{x}{2}\bigg).
\end{equation*}
From the monotonicity property we thus deduce that, for $\nu>\frac{1}{2}$,
\begin{equation*}\frac{\mathbf{L}_{\nu-1}(x)}{\mathbf{L}_{\nu}(x)}-\frac{2\nu}{x}-\frac{2b_\nu(x)}{x}<\frac{1}{\tanh\big(\frac{1}{2}x\big)}-\frac{1}{x}-\frac{2b_{\frac{1}{2}}(x)}{x},
\end{equation*}
whence on rearranging we obtain (\ref{bibidt}), as required.
\end{proof}

\begin{remark}Inequality (\ref{lower1}) has a larger range of validity than the lower bound of (\ref{upper1}), but is outperformed by (\ref{upper1}) for all $x>0$ if $\nu\geq0$.  However, as we shall see shortly, in some situations, for reasons of simplicity, (\ref{lower1}) may be preferable to the lower bound of (\ref{upper1}).  Inequality (\ref{bibidt}) also improves inequality (\ref{lowtanh}) for all $x>0$ if $\nu>\frac{1}{2}$.  Together with the upper bound (\ref{bpineq3}) it forms a two-sided inequality for $\mathbf{L}_\nu(x)/\mathbf{L}_{\nu-1}(x)$, involving hyperbolic functions, that is exact for $\nu=\frac{1}{2}$. 
\end{remark}

We end this section by noting that it is possible to obtain further bounds for the ratio $\mathbf{L}_\nu(x)/\mathbf{L}_{\nu-1}(x)$ from the relation
\begin{equation}\label{iteqn1}\frac{\mathbf{L}_{\nu}(x)}{\mathbf{L}_{\nu-1}(x)}=\frac{1}{\displaystyle \frac{2\nu}{x}+\frac{2b_\nu(x)}{x}+\frac{\mathbf{L}_{\nu+1}(x)}{\mathbf{L}_{\nu}(x)}}.
\end{equation}
An analogous relation for the ratio $I_\nu(x)/I_{\nu-1}(x)$ has been used by \cite{amos74,rs16,segura} to obtain a sequence of iteratively refined upper and lower bounds that converge to the ratio $I_\nu(x)/I_{\nu-1}(x)$.  We do not undertake such an investigation in this paper, and we contend ourselves with the following simple illustration of the approach.

\begin{corollary}\label{thmr2}For $x>0$,
\begin{align}\label{upperlower}\frac{\mathbf{L}_{\nu}(x)}{\mathbf{L}_{\nu-1}(x)}<\frac{x}{\nu-1+2b_\nu(x)-b_{\nu+1}(x)+\sqrt{\big(\nu+1+b_{\nu+1}(x)\big)^2+x^2}}, \quad \nu\geq 0.
\end{align}
\end{corollary}

\begin{proof}Applying inequality (\ref{lower1}) to the relation (\ref{iteqn1}) gives the inequality
\begin{align}\frac{\mathbf{L}_{\nu}(x)}{\mathbf{L}_{\nu-1}(x)}&<\frac{1}{\displaystyle\frac{2\nu}{x}+\frac{2b_\nu(x)}{x}+\frac{x}{\nu+1+b_{\nu+1}(x)+\sqrt{(\nu+1+b_{\nu+1}(x))^2+x^2}}} \nonumber\\
\label{jko}&=\frac{1}{\displaystyle\frac{2\nu}{x}+\frac{2b_\nu(x)}{x}+\frac{-\nu-1-b_{\nu+1}(x)+\sqrt{(\nu+1+b_{\nu+1}(x))^2+x^2}}{x}} \\
&=\frac{x}{\displaystyle \nu-1+2b_\nu(x)-b_{\nu+1}(x)+\sqrt{(\nu+1+b_{\nu+1}(x))^2+x^2}}.\nonumber
\end{align} 
\end{proof}

\begin{remark}On applying the upper bound of (\ref{upper1}) of Corollary \ref{bghj} to the relation (\ref{iteqn1}) we recover the lower bound of (\ref{upper1}), through an alternative method.

We could have used the lower bound of (\ref{upper1}), instead of the lower bound (\ref{lower1}), to obtain an upper bound for $\mathbf{L}_{\nu}(x)/\mathbf{L}_{\nu-1}(x)$ that is less than (\ref{upperlower}) for all $x>0$ and $\nu\geq0$.  However, this bound would not take such a simple form, because the form of the lower bound of (\ref{upper1}) would not allow for such a neat simplification as the one used to obtain the equality (\ref{jko}).

A straightforward asymptotic analysis shows that inequality (\ref{upperlower}) improves on the upper bound of (\ref{upper1}) in the limit $x\downarrow0$ (in fact the relative error in approximating $\mathbf{L}_{\nu}(x)/\mathbf{L}_{\nu-1}(x)$ is 0 in this limit), whereas the reverse is true in the limit $x\rightarrow\infty$ (the first two terms in the $x\rightarrow\infty$ expansion are given by $1-(\nu-1)/x$).  Letting $x_\nu^*>0$ denote the point at which the two inequalities are equal, we used \emph{Mathematica} to find that $x_1^*=5.34$, $x_{2.5}^*=8.42$, $x_5^*=14.9$.  For $\nu\geq5$, we find that $x_\nu^*\approx 2\sqrt{\nu(2\nu+1)}$ (for example, $2\sqrt{5\times11}=14.83$).  This follows from setting $b_\nu(x)$ to be equal to 0 in (\ref{upperlower}) and then solving $\sqrt{(\nu+1)^2+(x_\nu^*)^2}\approx1/2+\sqrt{(\nu+1/2)^2+(x_\nu^*)^2}$, which, on account of the exponential decay of $b_\nu(x)$, is a reasonable approximation.  
\end{remark}

\section{Further bounds for modified Struve functions of the first kind and their ratios}\label{sec3}

In this section, we apply the bounds of Section \ref{sec2} for the ratio $\mathbf{L}_{\nu}(x)/\mathbf{L}_{\nu-1}(x)$ to obtain further functional inequalities for the modified Struve function of the first kind. 

\subsection{Bounds for the condition numbers}\label{sec3.2}

We shall follow the notation of \cite{segura} and write $C\big(\mathbf{L}_\nu(x)\big)=x\mathbf{L}'_\nu(x)/\mathbf{L}_\nu(x)$ and $C\big(I_\nu(x)\big)=xI'_\nu(x)/I_\nu(x)$.  These are positive quantities if $\nu\geq-1$ and $\nu\geq0$, respectively.  See also \cite{segura} for comments regarding the utility of condition numbers $C\big(f(x)\big)=|xf'(x)/f(x)|$ in comparing functions.  The first inequalities for $C\big(I_\nu(x)\big)$, due to \cite{g32}, were motivated by a problem in wave mechanics. 

From the relations (\ref{struveid1}) and (\ref{struveid2}) we obtain the relations
\begin{align}\label{star60}C\big(\mathbf{L}_\nu(x)\big)&=\frac{x\mathbf{L}_{\nu-1}(x)}{\mathbf{L}_\nu(x)}-\nu, \\
\label{star61}C\big(\mathbf{L}_\nu(x)\big)&=\frac{x\mathbf{L}_{\nu+1}(x)}{\mathbf{L}_\nu(x)}+\nu+2b_\nu(x),
\end{align}
and thus bounds on the ratio $\mathbf{L}_\nu(x)/\mathbf{L}_{\nu-1}(x)$ immediately lead to bounds on the condition number $C\big(\mathbf{L}_\nu(x)\big)$.  (For the modified Bessel function $I_\nu(x)$, the same relations hold but without the $2b_\nu(x)$ term.)
This approach was used by \cite{bp14} to obtain the lower bounds
\begin{equation*}C\big(\mathbf{L}_\nu(x)\big)>\nu+1, \quad \nu>-\tfrac{3}{2}, \quad C\big(\mathbf{L}_\nu(x)\big)>x-\nu, \quad \nu\geq-\tfrac{1}{2},
\end{equation*}
and the following improvement of the second bound (this follows from inequality (\ref{bpineq3})):
\begin{equation*}C\big(\mathbf{L}_\nu(x)\big)\geq\frac{x\sinh(x)}{\cosh(x)-1}-\nu=x\coth\bigg(\frac{x}{2}\bigg)-\nu,
\end{equation*}
with equality if and only if $\nu=\frac{1}{2}$.  Also, rearranging inequality (2.5) of \cite{bp14} and using the notation of this paper gives the upper bound
\begin{equation}\label{apti}C\big(\mathbf{L}_\nu(x)\big)<\sqrt{x^2+\nu^2+2(2\nu+1) b_\nu(x)}, \quad \nu>-\tfrac{3}{2}.
\end{equation}

In the following theorem, we given a two-sided inequality for $C\big(\mathbf{L}_\nu(x)\big)$ in terms of $C\big(I_\nu(x)\big)$, which parallels the two-sided inequality (\ref{firstcor1}) of Theorem \ref{thmil} in that it allows one to make use of the literature on bounds for $C\big(I_\nu(x)\big)$ (see \cite{amos74,b092,g32,ln10,p99,pm50,segura}) to bound $C\big(\mathbf{L}_\nu(x)\big)$.  We also obtain a number of lower and upper bounds which complement inequality (\ref{apti}).

\begin{theorem}The following inequalities hold:

\vspace{2mm}

\noindent (i) For $x>0$,
\begin{equation}\label{condlo}C\big(I_\nu(x)\big)<C\big(\mathbf{L}_\nu(x)\big)<C\big(I_\nu(x)\big)+2b_\nu(x),
\end{equation}
where the lower bound is valid for $\nu\geq\frac{1}{2}$ and the upper bound is valid for $\nu\geq-\frac{1}{2}$.

\vspace{2mm}

\noindent (ii) For $x>0$,
\begin{align}\label{star10}\sqrt{\big(\nu-\tfrac{1}{2}\big)^2+x^2}-\tfrac{1}{2}<C\big(\mathbf{L}_\nu(x)\big)<\sqrt{\big(\nu+b_\nu(x)\big)^2+x^2}+b_\nu(x),
\end{align}
where the lower bound holds for $\nu\geq\frac{1}{2}$ and the upper bound holds for $\nu\geq-\frac{1}{2}$;
\begin{align}&\sqrt{\big(\nu+1+b_{\nu+1}(x)\big)^2+x^2}+2b_\nu(x)-b_{\nu+1}(x)-1<C\big(\mathbf{L}_\nu(x)\big)\nonumber \\
\label{star11}&\quad\quad\quad\quad\quad<\sqrt{\big(\nu+\tfrac{1}{2}\big)^2+x^2}+2b_{\nu}(x)-\tfrac{1}{2},
\end{align}
where the lower bound holds for $\nu\geq-1$ and the upper bound holds for $\nu\geq-\frac{1}{2}$;
\begin{align}\label{star40}C\big(\mathbf{L}_\nu(x)\big)>\nu+2b_{\nu}(x)+\frac{x^2}{\nu+\frac{1}{2}+2b_{\nu+1}(x)+\sqrt{\big(\nu+\frac{3}{2}\big)^2+x^2}},
\end{align}
which is valid for $\nu\geq-1$.
\end{theorem}

\begin{proof}\noindent (i) From the relations (\ref{star60}) and (\ref{star61}) for $\mathbf{L}_\nu(x)$, as well as the corresponding ones for $I_\nu(x)$, and the upper bound of inequality (\ref{firstcor1}) of Theorem \ref{thmil}, we have, for $\nu\geq\frac{1}{2}$,
\begin{align*}\frac{x\mathbf{L}_\nu'(x)}{\mathbf{L}_\nu(x)}=\frac{x\mathbf{L}_{\nu-1}(x)}{\mathbf{L}_\nu(x)}-\nu>\frac{xI_{\nu-1}(x)}{I_\nu(x)}-\nu=\frac{xI_\nu'(x)}{I_\nu(x)},
\end{align*}
and, for $\nu\geq-\frac{1}{2}$,
\begin{align*}\frac{x\mathbf{L}_\nu'(x)}{\mathbf{L}_\nu(x)}=\frac{x\mathbf{L}_{\nu+1}(x)}{\mathbf{L}_\nu(x)}+\nu+2b_\nu(x)<\frac{xI_{\nu+1}(x)}{I_\nu(x)}+\nu+2b_\nu(x)=\frac{xI_\nu'(x)}{I_\nu(x)}+2b_\nu(x).
\end{align*}

\vspace{2mm}

\noindent (ii) We obtain the double inequality (\ref{star10}) by combining the upper bound of (\ref{upper1}) and inequality (\ref{lower1}) with (\ref{star60}).  The double inequality (\ref{star11}) follows from combining inequality (\ref{lower1}) and the upper bound of (\ref{upper1}) with (\ref{star61}).  Finally, inequality (\ref{star40}) follows from applying the lower bound of (\ref{upper1}) to  (\ref{star61}).

Alternatively, one can obtain the lower bound of (\ref{star10}) and the upper bound of (\ref{star11}) by combining part (i) with the following inequality (see inequalities (71) and (72) of \cite{segura}):
\[\sqrt{\big(\nu-\tfrac{1}{2}\big)^2+x^2}-\tfrac{1}{2}<C\big(I_\nu(x)\big)<\sqrt{\big(\nu+\tfrac{1}{2}\big)^2+x^2}-\tfrac{1}{2},\]
where the lower bound holds for $\nu\geq\frac{1}{2}$ and the upper bound holds for $\nu\geq-\frac{1}{2}$.
\end{proof}

\begin{remark}Since the lower bound of (\ref{upperlower}) is greater than the lower bound (\ref{lower1}), it follows that the lower bound (\ref{star40}) for the condition number $C\big(\mathbf{L}_\nu(x)\big)$ is greater than the lower bound (\ref{star11}).  Also, the upper bound (\ref{star11}) is less than the upper bound (\ref{star10}).  However, the comparison between the lower bounds (\ref{star10}) and (\ref{star40}), and the upper bounds (\ref{apti}) and (\ref{star11}) is more involved.  Denote the lower bounds (\ref{star10}) and (\ref{star40}) by $l_\nu^{c,1}(x)$ and $l_\nu^{c,2}(x)$, respectively, and the upper bounds (\ref{apti}) and (\ref{star11}) by $u_\nu^{c,1}(x)$ and $u_\nu^{c,2}(x)$, respectively.  We shall compare their asymptotic behaviour as $x\downarrow0$ and $x\rightarrow\infty$.  For reference, using the asymptotic formulas (\ref{ltend0}) and (\ref{linfty}) gives that
\begin{align*}C\big(\mathbf{L}_\nu(x)\big)\sim \nu+1+\frac{2x^2}{3(2\nu+3)},\:\: x\downarrow0, \quad C\big(\mathbf{L}_\nu(x)\big)\sim x-\frac{1}{2}+\frac{4\nu^2-1}{8x}, \:\: x\rightarrow\infty.
\end{align*}
From the asymptotic formulas (\ref{bnu0}) and (\ref{bnu1}) we obtain
\begin{align*}&l_\nu^{c,1}(x)\sim \nu-1+\frac{x^2}{2\nu-1},\:\: x\downarrow0,\: (\nu>-\tfrac{1}{2}), \quad l_\nu^{c,1}(x)\sim x-\frac{1}{2}+\frac{(2\nu-1)^2}{8x}, \:\: x\rightarrow\infty, \\
&l_\nu^{c,2}(x)\sim \nu+1+\frac{2x^2}{3(2\nu+3)},\:\: x\downarrow0, \quad l_\nu^{c,2}(x)\sim x-\frac{1}{2}+\bigg(\frac{(2\nu-1)^2}{8}-1\bigg)\frac{1}{x}, \:\: x\rightarrow\infty, 
\end{align*}
\begin{align*}
&u_\nu^{c,1}(x)\sim \nu+1+\frac{2(\nu+2)x^2}{3(\nu+1)(2\nu+3)},\:\: x\downarrow0, \quad u_\nu^{c,1}(x)\sim x+\frac{\nu^2}{2x}, \:\: x\rightarrow\infty, \\
&u_\nu^{c,2}(x)\sim \nu+1+\frac{(10\nu+17)x^2}{6(2\nu+1)(2\nu+3)},\:\: x\downarrow0, \quad u_\nu^{c,2}(x)\sim x-\frac{1}{2}+\frac{(2\nu+1)^2}{8x}, \:\: x\rightarrow\infty.
\end{align*}
Other than $l_\nu^{c,1}(x)$ in the limit $x\downarrow0$, the bounds are tight as $x\downarrow0$ and $x\rightarrow\infty$.  Also, numerical experiments suggest that, for all $\nu$ in the ranges of validity, $l_\nu^{c,1}(x)<l_\nu^{c,2}(x)$ for all $x\in(0,x_\nu^*)$ and $l_\nu^{c,1}(x)>l_\nu^{c,2}(x)$ for all $x>x_\nu^*$, for some $x_\nu^*>0$, and $u_\nu^{c,1}(x)<u_\nu^{c,2}(x)$ for all $x\in(0,x_\nu^{**})$ and $u_\nu^{c,1}(x)>u_\nu^{c,2}(x)$ for all $x>x_\nu^{**}$, for some $x_\nu^{**}>0$.
\end{remark}

\subsection{Upper and lower bounds for the ratio $\mathbf{L}_\nu(x)/\mathbf{L}_\nu(y)$ and the modified Struve function $\mathbf{L}_\nu(x)$}\label{sec3.3}

The results of Section \ref{sec3.2} have the following immediate application.  Integrating $A_\nu(t)<C\big(\mathbf{L}_\nu(t)\big)<B_\nu(t)$ gives the two-sided inequality
\begin{equation}\label{abcineq}\exp\bigg(-\int_x^y\frac{A_\nu(t)}{t}\,\mathrm{d}t\bigg)<\frac{\mathbf{L}_\nu(x)}{\mathbf{L}_\nu(y)}<\exp\bigg(-\int_x^y\frac{B_\nu(t)}{t}\,\mathrm{d}t\bigg).
\end{equation}
This approach was used by \cite{bp14} to prove that, for $0<x<y$,
\begin{equation}\label{nvnvj}\frac{\mathbf{L}_\nu(x)}{\mathbf{L}_\nu(y)}<\bigg(\frac{x}{y}\bigg)^{\nu+1}, \quad \nu>-\tfrac{3}{2} \quad \text{and} \quad \frac{\mathbf{L}_\nu(x)}{\mathbf{L}_\nu(y)}<\mathrm{e}^{x-y}\bigg(\frac{y}{x}\bigg)^\nu, \quad \nu\geq\tfrac{1}{2},
\end{equation}
where the first inequality of (\ref{nvnvj}) was proved for $\nu>-\frac{1}{2}$ by \cite{jn98}, and it was also shown by \cite{bp14} that the second inequality of (\ref{nvnvj}) can be improved to
\begin{equation}\label{bpineq5}\frac{\mathbf{L}_\nu(x)}{\mathbf{L}_\nu(y)}\leq\bigg(\frac{\cosh(x)-1}{\cosh(y)-1}\bigg)\bigg(\frac{y}{x}\bigg)^\nu,
\end{equation}
with equality if and only if $\nu=\frac{1}{2}$. 

We do not further explore combining the bounds of Section \ref{sec3.2} with (\ref{abcineq}), but instead note that, in a similar manner, one can integrate the relations
\begin{align*}\frac{\mathbf{L}_\nu'(t)}{\mathbf{L}_\nu(t)}=\frac{\mathbf{L}_{\nu-1}(t)}{\mathbf{L}_\nu(t)}-\frac{\nu}{t}, \quad
\frac{\mathbf{L}_\nu'(t)}{\mathbf{L}_\nu(t)}=\frac{\mathbf{L}_{\nu+1}(t)}{\mathbf{L}_\nu(t)}+\frac{\nu}{t}+\frac{a_\nu(t)}{\mathbf{L}_\nu(t)}
\end{align*}
between $x$ and $y$ to obtain
\begin{align}\label{star30}\frac{\mathbf{L}_\nu(x)}{\mathbf{L}_\nu(y)}&=\bigg(\frac{y}{x}\bigg)^\nu\exp\bigg(-\int_x^y\frac{\mathbf{L}_{\nu-1}(t)}{\mathbf{L}_\nu(t)}\,\mathrm{d}t\bigg), \\
\label{star31}\frac{\mathbf{L}_\nu(x)}{\mathbf{L}_\nu(y)}&=\bigg(\frac{x}{y}\bigg)^\nu\exp\bigg(-2\int_x^y\frac{b_\nu(t)}{t}\,\mathrm{d}t\bigg)\exp\bigg(-\int_x^y\frac{\mathbf{L}_{\nu+1}(t)}{\mathbf{L}_\nu(t)}\,\mathrm{d}t\bigg).
\end{align}
We combine these formulas and the bounds of Section \ref{sec2} to prove the following theorem.

\begin{theorem}\label{thmalm}The following inequalities hold:

\vspace{2mm}

\noindent (i) For $0<x<y$,
\begin{equation}\label{thmalm1}\frac{x}{y}\sqrt{\frac{3(2\nu+3)+y^2}{3(2\nu+3)+x^2}}\frac{I_\nu(x)}{I_\nu(y)}<\frac{\mathbf{L}_\nu(x)}{\mathbf{L}_\nu(y)}<\frac{I_\nu(x)}{I_\nu(y)},
\end{equation}
where the lower bound holds for $\nu\geq-\frac{1}{2}$ and the upper bound holds for $\nu\geq\frac{1}{2}$.

\vspace{2mm}

\noindent (ii) Let $\nu\geq-\frac{1}{2}$.  Then, for $0<x<y$,
\begin{align}&\frac{\mathrm{e}^{\sqrt{(\nu+1/2)^2+x^2}}}{\mathrm{e}^{\sqrt{(\nu+1/2)^2+y^2}}}\bigg(\frac{x}{y}\bigg)^{\nu+1}\sqrt{\frac{3(2\nu+3)+y^2}{3(2\nu+3)+x^2}}\Bigg(\frac{\nu+\frac{1}{2}+\sqrt{(\nu+\frac{1}{2})^2+y^2}}{\nu+\frac{1}{2}+\sqrt{(\nu+\frac{1}{2})^2+x^2}}\Bigg)^{\nu+\frac{1}{2}}<\frac{\mathbf{L}_\nu(x)}{\mathbf{L}_\nu(y)}<\nonumber\\
\label{thmu9}&\quad<\frac{\mathrm{e}^{\sqrt{(\nu+3/2)^2+x^2}}}{\mathrm{e}^{\sqrt{(\nu+3/2)^2+y^2}}}\frac{\tanh\big(\frac{1}{2}x\big)}{\tanh\big(\frac{1}{2}y\big)}\bigg(\frac{x}{y}\bigg)^{\nu}\Bigg(\frac{\nu+\frac{3}{2}+\sqrt{(\nu+\frac{3}{2})^2+y^2}}{\nu+\frac{3}{2}+\sqrt{(\nu+\frac{3}{2})^2+x^2}}\Bigg)^{\nu+\frac{3}{2}}.
\end{align}

\vspace{2mm}

\noindent (iii) Let $\nu\geq-\frac{1}{2}$.  Then, for $x>0$,
\begin{align}\label{coru1}&\frac{\mathrm{e}^{\sqrt{(\nu+3/2)^2+x^2}-\nu-3/2}}{\sqrt{\pi}2^{\nu-1}\Gamma(\nu+\frac{3}{2})}x^\nu\tanh\Big(\frac{x}{2}\Big)\Bigg(\frac{2\nu+3}{\nu+\frac{3}{2}+\sqrt{\big(\nu+\frac{3}{2}\big)^2+x^2}}\Bigg)^{\nu+\frac{3}{2}}<\mathbf{L}_\nu(x)<\nonumber \\
&\quad<\frac{\mathrm{e}^{\sqrt{(\nu+1/2)^2+x^2}-\nu-1/2}}{\sqrt{\pi}2^{\nu}\Gamma(\nu+\frac{3}{2})}x^{\nu+1}\sqrt{\frac{3(2\nu+3)}{3(2\nu+3)+x^2}}\Bigg(\frac{2\nu+1}{\nu+\frac{1}{2}+\sqrt{\big(\nu+\frac{1}{2}\big)^2+x^2}}\Bigg)^{\nu+\frac{1}{2}}.
\end{align}
\end{theorem}

\begin{proof}(i) We begin by noting two integral formulas.  From the relations  $I_\nu'(t)=I_{\nu-1}(t)-\frac{\nu}{t}I_\nu(t)$ and $I_\nu'(t)=I_{\nu+1}(t)+\frac{\nu}{t}I_\nu(t)$, we have
\begin{align*}\int\frac{I_{\nu}(t)}{I_{\nu-1}(t)}\,\mathrm{d}t&=\int\bigg(\frac{I_{\nu-1}'(t)}{I_{\nu-1}(t)}-\frac{\nu-1}{t}\bigg)\,\mathrm{d}t=\log\big(I_{\nu-1}(t)\big)-(\nu-1)\log(t),\\
\int\frac{I_{\nu-1}(t)}{I_{\nu}(t)}\,\mathrm{d}t&=\int\bigg(\frac{I_{\nu}'(t)}{I_{\nu}(t)}+\frac{\nu}{t}\bigg)\,\mathrm{d}t=\log\big(I_{\nu}(t)\big)+\nu\log(t).
\end{align*}

Now, let us prove the upper bound.  By the upper bound of (\ref{firstcor1}), we have, for $\nu\geq\frac{1}{2}$,
\begin{align*}\int_x^y\frac{\mathbf{L}_{\nu}(t)}{\mathbf{L}_{\nu-1}(t)}\,\mathrm{d}t>\int_x^y\frac{I_{\nu}(t)}{I_{\nu-1}(t)}\,\mathrm{d}t=\log\bigg(\frac{I_{\nu}(y)}{I_{\nu}(x)}\bigg)+\nu\log\bigg(\frac{y}{x}\bigg).
\end{align*}
Combining this inequality with (\ref{star30}) yields the upper bound.  For the lower bound, we use the upper bound of (\ref{firstcor1}) to obtain, for $\nu\geq-\frac{1}{2}$,
\begin{align*}\int_x^y\frac{\mathbf{L}_{\nu+1}(t)}{\mathbf{L}_{\nu}(t)}\,\mathrm{d}t<\int_x^y\frac{I_{\nu+1}(t)}{I_{\nu}(t)}\,\mathrm{d}t=\log\bigg(\frac{I_{\nu}(y)}{I_{\nu}(x)}\bigg)-\nu\log\bigg(\frac{y}{x}\bigg).
\end{align*}
Also, by inequality (\ref{bcrude2}),
\begin{equation}\label{mncz}2\int_x^y\frac{b_\nu(t)}{t}\,\mathrm{d}t<\int_x^y\frac{1}{t\big(1+\frac{1}{3(2\nu+3)}t^2\big)}\,\mathrm{d}t=\log\Bigg(\frac{y}{x}\sqrt{\frac{3(2\nu+3)+x^2}{3(2\nu+3)+y^2}}\Bigg).
\end{equation}
On combining these inequalities with (\ref{star31}) we obtain the lower bound.

\vspace{2mm}

\noindent (ii) Let $\nu\geq-\frac{1}{2}$. We first prove the upper bound.  From the lower bound of (\ref{upper1}) and the fact that $b_\nu(t)<\frac{1}{2}$ we have
\begin{align}\label{intlp}&\int_x^y\frac{\mathbf{L}_{\nu+1}(t)}{\mathbf{L}_\nu(t)}\,\mathrm{d}t>\int_x^y\frac{t}{\nu+\frac{1}{2}+2b_{\nu+1}(t)+\sqrt{\big(\nu+\frac{3}{2}\big)^2+t^2}}\,\mathrm{d}t \\
&\quad\quad>\int_x^y\frac{t}{\nu+\frac{3}{2}+\sqrt{\big(\nu+\frac{3}{2}\big)^2+t^2}}\,\mathrm{d}t \nonumber \\
&\quad\quad=\sqrt{(\nu+\tfrac{3}{2})^2+y^2}-\sqrt{(\nu+\tfrac{3}{2})^2+x^2} +(\nu+\tfrac{3}{2})\log\Bigg(\frac{\nu+\frac{3}{2}+\sqrt{(\nu+\frac{3}{2})^2+x^2}}{\nu+\frac{3}{2}+\sqrt{(\nu+\frac{3}{2})^2+y^2}}\Bigg). \nonumber 
\end{align}
Also, from the lower bound of (\ref{star5}), we obtain
\begin{equation*}2\int_x^y\frac{b_\nu(t)}{t}\,\mathrm{d}t>\int_x^y\mathrm{csch}(t)\,\mathrm{d}t=\log\Bigg(\frac{\tanh\big(\frac{1}{2}x\big)}{\tanh\big(\frac{1}{2}y\big)}\Bigg).
\end{equation*}
On combining these inequalities with (\ref{star31}) we obtain the upper bound, as required.

Now, we prove the lower bound.  From the upper bound of (\ref{upper1}) we have
\begin{align*}&\int_x^y\frac{\mathbf{L}_{\nu+1}(t)}{\mathbf{L}_\nu(t)}\,\mathrm{d}t<\int_x^y\frac{t}{\nu+\frac{1}{2}+\sqrt{(\nu+\frac{1}{2})^2+t^2}}\,\mathrm{d}t \\
&\quad\quad=\sqrt{(\nu+\tfrac{1}{2})^2+y^2}-\sqrt{(\nu+\tfrac{1}{2})^2+x^2}+(\nu+\tfrac{1}{2})\log\Bigg(\frac{\nu+\frac{1}{2}+\sqrt{(\nu+\frac{1}{2})^2+x^2}}{\nu+\frac{1}{2}+\sqrt{(\nu+\frac{1}{2})^2+y^2}}\Bigg).
\end{align*}
On combining this bound and inequality (\ref{mncz}) with (\ref{star31}) we obtain the lower bound, as required.

\vspace{2mm}

\noindent (iii) Let $x\downarrow0$ in (\ref{thmu9}) and use the limits $\lim_{x\downarrow0}\frac{\mathbf{L}_\nu(x)}{x^{\nu+1}}=\frac{1}{\sqrt{\pi}2^{\nu}\Gamma(\nu+\frac{3}{2})}$ and $\lim_{x\downarrow0}\frac{1}{x}\tanh(\frac{x}{2})=\frac{1}{2}$.  Then replace $y$ by $x$.
\end{proof}

\begin{remark}The double inequality in part (i) of Theorem \ref{thmalm} allows one to take advantage of the substantial literature (as given in the Introduction) on bounds for the ratio $I_\nu(x)/I_\nu(y)$.  For example, inequality (2.19) of \cite{baricz2} and  inequality (1.6) of \cite{jb96}, respectively, give that, for $0<x<y$,
\begin{align}\label{hbv}\frac{I_\nu(x)}{I_\nu(y)}&>\frac{\cosh(x)}{\cosh(y)}\bigg(\frac{x}{y}\bigg)^{\nu}, \quad \nu>-\tfrac{1}{2}, \\
\label{johl} \frac{I_\nu(x)}{I_\nu(y)}&>\mathrm{e}^{x-y}\bigg(\frac{y+\nu}{x+\nu}\bigg)^{\nu}\bigg(\frac{x}{y}\bigg)^{\nu}, \quad \nu\geq0,
\end{align}
where (\ref{johl}) is also valid for $\nu>-1$ (see \cite{jb96}) provided suitable restrictions are imposed on $x$ and $y$.  Combining (\ref{hbv}) with the lower bound of (\ref{thmalm1}) gives
\begin{equation*}\frac{\mathbf{L}_\nu(x)}{\mathbf{L}_\nu(y)}>\frac{\cosh(x)}{\cosh(y)}\bigg(\frac{x}{y}\bigg)^{\nu+1}\sqrt{\frac{3(2\nu+3)+y^2}{3(2\nu+3)+x^2}}, \quad \nu>-\tfrac{1}{2}, \:0<x<y,
\end{equation*}
which complements the upper bound (\ref{bpineq5}) that was proved by \cite{bp14}.  Now, from (\ref{johl}), we have
\begin{equation}\label{setit}\frac{\mathbf{L}_\nu(x)}{\mathbf{L}_\nu(y)}>\mathrm{e}^{x-y}\bigg(\frac{y+\nu}{x+\nu}\bigg)^{\nu}\bigg(\frac{x}{y}\bigg)^{\nu+1}\sqrt{\frac{3(2\nu+3)+y^2}{3(2\nu+3)+x^2}}, \quad \nu\geq0, \:0<x<y.
\end{equation}
Arguing as we did in the proof of part (iii) of Theorem \ref{thmalm}, we obtain the inequality
\begin{equation}\label{vivi}\mathbf{L}_\nu(x)<\frac{1}{\sqrt{\pi}2^\nu\Gamma(\nu+\frac{3}{2})}\bigg(\frac{\nu}{x+\nu}\bigg)^\nu\sqrt{\frac{3(2\nu+3)}{3(2\nu+3)+x^2}} x^{\nu+1}\mathrm{e}^x, \quad \nu\geq0, \:x>0.
\end{equation}
Inequalities (\ref{setit}) and (\ref{vivi}) are outperformed by the more complicated corresponding bounds from the two-sided inequalities (\ref{thmu9}) and (\ref{coru1}), respectively.
\end{remark}

\begin{remark}For fixed $y>0$, both the lower bound and upper bound of (\ref{thmu9}) are $O(x^{\nu+1})$, as $x\downarrow0$.  For fixed $x>0$, as $y\rightarrow\infty$, the lower bound is $O(y^{1/2}\mathrm{e}^{-y})$, which is the correct order (see (\ref{linfty})).  However, the upper bound is $O(y^{3/2}\mathrm{e}^{-y})$, as $y\rightarrow\infty$.  That the upper bound is not of the correct order as $y\rightarrow\infty$ can be at least partly traced back to the use of the inequality $b_\nu(t)<\frac{1}{2}$ in its derivation.  This inequality is tight as $t\downarrow0$, but very crude for large $t$ (see (\ref{bnu1})).  Using the refined inequality (\ref{bcrude2}) in the derivation of the lower bound (\ref{thmu9}) enabled us to obtain the correct order as $y\rightarrow\infty$, but using this inequality to bound the integral (\ref{intlp}) leads to a considerably less tractable integral.  

To the best knowledge of this author, our lower bound for the ratio $\mathbf{L}_\nu(x)/\mathbf{L}_\nu(y)$ is the first such bound to appear in the literature.  The upper bound improves on inequality (\ref{nvnvj}) for all $\nu\geq-\frac{1}{2}$.  The comparison between inequality (\ref{bpineq5}) and our upper bound is more involved.  Denote these bounds by $u_\nu^{d,1}(x,y)$ and $u_\nu^{d,2}(x,y)$, respectively.  In fact, $u_{\frac{1}{2}}^{d,1}(x,y)=\mathbf{L}_{\frac{1}{2}}(x)/\mathbf{L}_{\frac{1}{2}}(y)$.  For fixed $y$, as $x\downarrow0$, we have $u_\nu^{d,1}(x,y)=O(x^{2-\nu})$ and therefore, for `small' $x$, $u_\nu^{d,2}(x,y)$ outperforms $u_\nu^{d,1}(x,y)$ for all $\nu>-\frac{1}{2}$.  For fixed $x$, as $y\rightarrow\infty$, $u_\nu^{d,2}(x,y)=O(y^{3/2}\mathrm{e}^{-y})$  and $u_\nu^{d,1}(x,y)=O(y^{\nu}\mathrm{e}^{-y})$, and so, in this regime, $u_\nu^{d,2}(x,y)\gg u_\nu^{d,1}(x,y)$ if $\nu<\frac{3}{2}$ and $u_\nu^{d,2}(x,y)\ll u_\nu^{d,1}(x,y)$ if $\nu>\frac{3}{2}$.  

\end{remark}

\begin{remark}\label{rem12}The lower bound and upper bounds of (\ref{coru1}) for $\mathbf{L}_\nu(x)$ are both tight in the limit $x\downarrow0$.  As $x\rightarrow\infty$, the upper bound is $O(x^{-1/2}\mathrm{e}^x)$ and the lower bound is $O(x^{-3/2}\mathrm{e}^x)$, whereas $\mathrm{L}_\nu(x)=O(x^{-1/2}\mathrm{e}^x)$.  That the lower bound is not of the correct order as $x\rightarrow\infty$ is a consequence of the fact that the upper bound of (\ref{thmu9}) is not of the correct order as $y\rightarrow\infty$, as discussed in the previous remark.

We examine the upper bound of (\ref{coru1}) in further detail.  Denote this bound by $u_\nu^{e,1}(x)$.  It has the following asymptotic behaviour:
\begin{align*}u_\nu^{e,1}(x)&\sim\frac{x^{\nu+1}}{\sqrt{\pi}2^{\nu}\Gamma(\nu+\frac{3}{2})}\bigg(1+\frac{2(\nu+2)x^2}{3(2\nu+1)(2\nu+3)}\bigg), \quad x\downarrow0, \\
u_\nu^{e,1}(x) &\sim a_\nu\frac{\mathrm{e}^x}{\sqrt{x}}, \quad x\rightarrow\infty, 
\end{align*}
where 
\[a_\nu=\sqrt{\frac{12}{\pi}}\frac{\sqrt{\nu+\frac{3}{2}}}{\Gamma(\nu+\frac{3}{2})}(\nu+\tfrac{1}{2})^{\nu+1/2}\mathrm{e}^{-\nu-1/2},\]
which, by an application of Stirling's inequality \cite[5.6.1]{olver}, can be bounded by
\begin{equation}\label{aplm}\frac{\sqrt{6}}{\pi}\sqrt{\frac{2\nu+3}{2\nu+1}}\mathrm{e}^{-\frac{1}{6(2\nu+1)}}<a_\nu<\frac{\sqrt{6}}{\pi}\sqrt{\frac{2\nu+3}{2\nu+1}}, \quad \nu>-\tfrac{1}{2}.
\end{equation}
We see that the second term in the $x\downarrow0$ expansion of $u_\nu^{e,1}(x)$ approaches that of $\mathbf{L}_\nu(x)$ as $\nu\rightarrow\infty$.  It is straightforward to show that there exists a constant $C>0$, independent of $\nu$, such that $a_\nu<C$ for all $\nu\geq-\frac{1}{2}$, which means that, whilst $a_\nu>\frac{1}{\sqrt{2\pi}}$ (recall that $\mathbf{L}_\nu(x)\sim\frac{1}{\sqrt{2\pi x}}\mathrm{e}^x$, as $x\rightarrow\infty$), the upper bound is always within an absolute constant multiple of $\mathbf{L}_\nu(x)$ for `large' $x$.  Numerical results (see Table \ref{table5}) obtained using \emph{Mathematica} support this analysis and suggest that; for fixed $x>0$, the relative error in approximating $\mathbf{L}_\nu(x)$ decreases as $\nu$ increases, and, for fixed $\nu$, the relative error increases from an initial value of 0 at $x=0$ up to a maximum $\sqrt{2\pi}a_\nu-1$ as $x$ increases.


Several simpler upper bounds for $\mathbf{L}_\nu(x)$ in terms of elementary functions were given in Section 3 of \cite{bp14}, although none of these bounds are of the correct order as $x\rightarrow\infty$.  In addition, \cite{bp14} obtained the following useful inequality:
\begin{equation}\label{bpstu}\mathbf{L}_\nu(x)<\frac{2\Gamma(\nu+2)}{\sqrt{\pi}\Gamma(\nu+\frac{3}{2})}I_{\nu+1}(x), \quad x>0, \:\nu>-\tfrac{1}{2}.
\end{equation}
One can thus exploit the extensive literature on inequalities for $I_\nu(x)$ to upper bound $\mathbf{L}_\nu(x)$.  For example, applying Theorem 2 of \cite{p99} to (\ref{bpstu}) gives the bound
\begin{align}\label{nearr}\mathbf{L}_\nu(x)<\frac{\sqrt{2}\Gamma(\nu+2)}{\pi\Gamma(\nu+\frac{3}{2})}\frac{\mathrm{e}^{\sqrt{x^2+(\nu+1)^2}+\frac{2}{\sqrt{x^2+(\nu+1)^2}}}}{(x^2+(\nu+1)^2)^{1/4}}\bigg(\frac{x}{\nu+1+\sqrt{x^2+(\nu+1)^2}}\bigg)^{\nu+1},
\end{align}
which holds for all $\nu>-\frac{1}{2}$ and $x>0$.  The upper bound (\ref{nearr}), which we denote by $u_\nu^{e,2}(x)$, is $O(x^{\nu+1})$, as $x\downarrow0$, and $O(x^{-1/2}\mathrm{e}^x)$, as $x\rightarrow\infty$, which, as is the case for $u_\nu^{e,1}(x)$, is in agreement with the asymptotic behaviour of $\mathbf{L}_\nu(x)$.  However, the multiplicative constant in the leading term of the $x\downarrow0$ expansion is larger than that of $\mathbf{L}_\nu(x)$.  As $x\rightarrow\infty$, $u_\nu^{e,2}(x)\sim b_\nu x^{-1/2}\mathrm{e}^x$, where $b_\nu=\frac{\sqrt{2}\Gamma(\nu+2)}{\pi\Gamma(\nu+\frac{3}{2})}$.  Unlike $a_\nu$, the constant $b_\nu$ increases at rate $\sqrt{\nu}$ as $\nu\rightarrow\infty$ (see \cite[5.6.4]{olver}).  However, for small enough $\nu$ we have $b_\nu<a_\nu$; we used \emph{Mathematica} to find that $b_\nu=a_\nu$ when $\nu=2.521$.  Further numerical experiments suggest that $u_\nu^{e,1}(x)<u_\nu^{e,2}(x)$ for all $x>0$ if $\nu>2.521$.  Some results are reported in Table \ref{table6}.  Notice that up to $x=100$ we have $u_{2.5}^{e,1}(x)<u_{2.5}^{e,2}(x)$, and then, as would be expected, this inequality is reversed for `large' values of $x$.

\begin{table}[h]
\centering
\caption{\footnotesize{Relative error in approximating $\mathbf{L}_{\nu}(x)$ by the upper bound of (\ref{coru1}).}}
\label{table5}
{\scriptsize
\begin{tabular}{|c|rrrrrrrrrr|}
\hline
 \backslashbox{$\nu$}{$x$}      &   0.5 &    1 &    2.5 &    5 &     10 & 15 & 25 & 50 & 100 & 200 \\
 \hline
0 & 0.0743 & 0.2403 & 0.8053 & 1.3722  & 1.7107 & 1.7994 & 1.8540 & 1.8839 & 1.8951 & 1.9000 \\
1 & 0.0163 & 0.0618 & 0.2928 & 0.6854  & 1.0716 & 1.2020 & 1.2914 & 1.3462 & 1.3690 & 1.3792 \\
2.5 & 0.0052 & 0.0204 & 0.1151 & 0.3523  & 0.7301 & 0.9026 & 1.0340 & 1.1214 & 1.1602 & 1.1782 \\
5 & 0.0017 & 0.0070 & 0.0431 & 0.1612     & 0.4601  & 0.6600 & 0.8388 & 0.9706 & 1.0333 & 1.0635 \\
10 & 0.0005 & 0.0021 & 0.0135 & 0.0582  & 0.2302 & 0.4133 & 0.6309 & 0.8216 & 0.9238 & 0.9762 \\  
  \hline
\end{tabular}}
\end{table}

\begin{table}[h]
\centering
\caption{\footnotesize{Relative error in approximating $\mathbf{L}_{\nu}(x)$ by the upper bound (\ref{nearr}).}}
\label{table6}
{\scriptsize
\begin{tabular}{|c|rrrrrrrrrr|}
\hline
 \backslashbox{$\nu$}{$x$}      &   0.5 &    1 &    2.5 &    5  & 10 & 15 & 25 & 50 & 100 & 200 \\
 \hline
0 & 5.3417 & 3.2145 & 1.1605 & 0.6502  & 0.4549 & 0.3931 & 0.3445 & 0.3086 & 0.2908 & 0.2820 \\
1 & 1.7475 & 1.5473 & 1.0183 & 0.7830  & 0.7437 & 0.7328 & 0.7207 & 0.7098 & 0.7039 & 0.7008 \\
2.5 & 0.8072 & 0.7908 & 0.7309 & 0.7459  & 0.9201 & 1.0127 & 1.0853 & 1.1374 & 1.1627 & 1.1751  \\
5 & 0.4167 & 0.4215 & 0.4563 & 0.5838    & 0.9410  & 1.1928 & 1.4306 & 1.6218 & 1.7208 & 1.7712 \\
10 & 0.2102 & 0.2151 & 0.2491 & 0.3664  & 0.7552 & 1.1596 & 1.6780 & 2.1815 & 2.4709 & 2.6250  \\  
  \hline
\end{tabular}}
\end{table}
\end{remark}

\subsection*{Acknowledgements}
The author is supported by a Dame Kathleen Ollerenshaw Research Fellowship.


\footnotesize

\begin{thebibliography}{99}
\addcontentsline{toc}{section}{References}

\bibitem{amos74} Amos, D. E. Computation of modified Bessel functions and their ratios. \emph{Math. Comput.}
$\mathbf{28}$ (1974), pp. 239--251.


\bibitem{b09} Baricz, \'{A}. Tight bounds for the generalized Marcum $Q$-function. \emph{J. Math. Anal. Appl.} $\mathbf{360}$ (2009), pp. 265--277.

\bibitem{b092} Baricz, \'{A}. On a product of modified Bessel functions. \emph{P. Amer. Math. Soc.} $\mathbf{137}$ (2009), pp. 189--193.

\bibitem{baricz2} Baricz, \'{A}.  Bounds for modified Bessel functions of the first and second kinds. \emph{P. Edinb. Math. Soc.} $\mathbf{53}$ (2010), pp. 575--599.

\bibitem{b15} Baricz, \'{A}. Bounds for Tur\'{a}nians of modified Bessel functions. \emph{Expo. Math.} $\mathbf{33}$ (2015), pp. 223--251.

\bibitem{bps17} Baricz, \'{A}., Ponnusamy, S. and Singh, S. Tur\'{a}n type inequalities for Struve functions. \emph{J. Math. Anal. Appl.} $\mathbf{445}$ (2017), pp. 971--984.

\bibitem{bp13} Baricz, \'{A}. and Pog\'any, T. K. Integral representations and summations of modified Struve function. \emph{Acta Math. Hung.} $\mathbf{141}$ (2013), pp. 254--281.

\bibitem{bp14} Baricz, \'{A}. and Pog\'any, T. K.  Functional inequalities for modified Struve functions. \emph{P. Roy. Soc. Edinb. A} $\mathbf{144}$ (2014), pp. 891--904.

\bibitem{bp142} Baricz, \'{A}. and Pog\'any, T. K.  Functional inequalities for modified Struve functions II. \emph{Math. Inequal. Appl.} $\mathbf{17}$ (2014), pp. 1387--1398.

\bibitem{bs09} Baricz, \'{A} and Sun, Y. New bounds for the generalized Marcum $Q$-function. \emph{IEEE Trans. Info. Th.} $\mathbf{55}$ (2009), pp. 3091--3100.


\bibitem{bord} Bordelon, D. J. Problem 72-15, inequalities for special functions. \emph{SIAM Rev.} $\mathbf{15}$ (1973), pp. 665--668.

\bibitem{g18} Gaunt, R. E. Inequalities for integrals of the modified Struve function of the first kind. \emph{Results Math.} $\mathbf{73}$:65 (2018).

\bibitem{ast07} Gil, A. Segura, J. and Temme, N. M. \emph{Numerical Methods for Special Functions}. SIAM, Philadelphia, 2007.

\bibitem{g32} Gronwall, T. H. An inequality for the Bessel functions of the first kind with imaginary argument. \emph{Ann. Math.} $\mathbf{33}$ (1932), pp. 275--278.

\bibitem{kg13} Hornik, K. and Gr\"{u}n, B. Amos-type bounds for modified Bessel function ratios. \emph{J. Math. Anal. Appl.} $\mathbf{408}$ (2013), pp. 91--101.

\bibitem{hw94} Hurley, W. G. and Wilcox, D. J. Calculation of leakage inductance in transformer windings. \emph{IEEE Trans. Power Electron.} $\mathbf{9}$ (1994), pp. 121--126.

\bibitem{ifantis} Ifantis, E. K. and Siafarikas, P. D.  Bounds for modified Bessel functions. \emph{Rend. Circ.
Mat. Palermo} $\mathbf{40}$ (1991), pp. 347--356.

\bibitem{il07}Ismail, M.E.H. and Laforgia, A. Monotonicity properties of determinants of special functions. \emph{Constr. Approx.} $\mathbf{26}$ (2007) pp. 1--9.

\bibitem{im78}  Ismail, M.E.H. and Muldoon, M. E. Monotonicity of the zeros of a cross-product of Bessel functions. \emph{SIAM J. Math. Anal.} $\mathbf{9}$ (1978), pp. 759--767.


\bibitem{jb96} Joshi, C. M. and Bissu, S. K. Some inequalities of Bessel and modified Bessel functions. \emph{J. Austral. Math. Soc. A} $\mathbf{50}$ (1991), pp. 333--342.

\bibitem{jbb} Joshi, C. M. and Bissu, S. K. Inequalities for some special functions. \emph{J. Comput. Appl. Math.} $\mathbf{69}$ (1996), pp. 251--259.

\bibitem{jn98} Joshi, C. M. and Nalwaya, S. Inequalities for modified Struve functions. \emph{J. Indian Math. Soc.} $\mathbf{65}$ (1998) pp. 49--57.

\bibitem{ks08} Khazron, P. A. and Selesnick, I. W. Bayesian estimation of Bessel K form random vectors in AWGN. \emph{IEEE Signal Process. Lett.} $\mathbf{15}$ (2008), pp. 261--264.

\bibitem{l91} Laforgia, A.  Bounds for modified Bessel functions. \emph{J. Comput. Appl. Math.} $\mathbf{34}$ (1991), pp. 263--267.

\bibitem{lm89} Laforgia A. and Mathis L. M. Bounds for Bessel functions. \emph{Rend. Circ. Mat. Palermo} $\mathbf{38}$ (1989), pp. 319--328.

\bibitem{ln10} Laforgia, A. and Natalini, P. Some Inequalities for Modified Bessel Functions. \emph{J. Inequal. Appl.} (2010), Art. ID 253035, 10 pp.

\bibitem{lorch} Lorch, L. Monotonicity of the zeros of a cross-product of Bessel functions. \emph{Methods Appl. Anal.} $\mathbf{1}$ (1994), pp. 75--80.

\bibitem{lbf03} Lushnikov, A. A., Bhatt, J. S. and Ford, I. J. Stochastic approach to chemical kinetics in ultrafine aerosols. \emph{J. Aerosol Sci.} $\mathbf{34}$ (2003) pp. 1117--1133.

\bibitem{mh69} Miles, J. W. and Huppert, H. E. Lee waves in a stratified flow. Part 4. Perturbation approximations. \emph{J. Fluid Mech.} $\mathbf{35}$ (1969), pp. 497--525.

\bibitem{olver} Olver, F. W. J., Lozier, D. W., Boisvert, R. F. and Clark, C. W.  \emph{NIST Handbook of Mathematical Functions.} Cambridge University Press, 2010.

\bibitem{nasell} N\r{a}sell, I.  Inequalities for Modified Bessel Functions.  \emph{Math. Comput.} $\mathbf{28}$ (1974), pp. 253--256.

\bibitem{nasell2} N\r{a}sell, I. Rational bounds for ratios of modified Bessel functions. \emph{SIAM J. Math. Anal.} $\mathbf{9}$ (1978), pp. 1--11. 

\bibitem{nh73} N\r{a}sell, I. and Hirsch, W. M. The transmission dynamics of schistosomiasis. \emph{Comm. Pure Appl. Math.} $\mathbf{26}$ (1973), pp. 395--453.

\bibitem{p99} Pal'tsev, B. V. Two-Sided Bounds Uniform in the Real Argument and the Index for Modified Bessel Functions. \emph{Math. Notes+} $\mathbf{65}$ (1999), pp. 571--581.

\bibitem{paris} Paris, R. B. An Inequality for the Bessel Function $J_\nu (\nu x)$. \emph{SIAM J. Math. Anal.} $\mathbf{15}$ (1984), pp. 203--205.

\bibitem{pm50} Phillips, R. S. and Malin, H. Bessel function approximations. \emph{Amer. J. Math.} $\mathbf{72}$ (1950) pp. 407--418.

\bibitem{ross73} Ross, D. K. Problem 72-15, inequalities for special functions. \emph{SIAM Rev.} $\mathbf{15}$ (1973), pp. 668--670.

\bibitem{rs16} Ruiz-Antol\'{i}n, D. and Segura. J. A new type of sharp bounds for ratios of modified Bessel functions. \emph{J. Math. Anal. Appl.} $\mathbf{443}$ (2016), pp. 1232--1246.

\bibitem{segura} Segura, J.  Bounds for ratios of modified Bessel functions and associated Tur\'{a}n-type inequalities.  \emph{J. Math. Anal. Appl.} $\mathbf{374}$ (2011), pp. 516--528.

\bibitem{s12} Segura, J. On bounds for solutions of monotonic first order difference-differential systems. \emph{J. Inequal. Appl.} $\mathbf{2012}$:65 (2012), 17 pp.

\bibitem{stoch} Shaked, M. and Shanthikumar, J. G. \emph{Stochastic Orders and their Applications.} Associated Press, 1994.

\bibitem{si95} Sitnik, S. M. Inequalities for Bessel functions. \emph{Dokl. Akad. Nauk SSSR} $\mathbf{340}$ (1995), pp. 29--32.

\bibitem{soni} Soni, R. P. On an inequality for modified Bessel functions. \emph{J. Math. Phys. Camb.} $\mathbf{44}$ (1965), pp. 406--407.

\bibitem{s84} Stephens, G. L. Scattering of plane waves by soft obstacles: anomalous diffraction theory for circular cylinders. \emph{Appl. Opt.} $\mathbf{23}$ (1984), pp. 954--959.

\bibitem{w55} Watkins, C. E. and Berman, J. H. On the kernel function of the integral equation relating the lift and downwash distributions of oscillating wings in supersonic flow. \emph{NACA Tech. Note} 1257 (1955), pp. 147--164.

\bibitem{w54} Watkins, C. E., Runyan, H. L. and Woolston, D. S. On the kernel function of the integral equation relating the lift and downwash distributions of oscillating finite wings in subsonic flow. \emph{NACA Tech. Note} 1234 (1954), pp. 703--718.


\end{thebibliography}
\end{document}